\documentclass{amsart}

\usepackage{amscd, amssymb, amsmath, amsthm}
\usepackage{array,caption}
\usepackage{mathabx} 
\usepackage{enumerate, latexsym, mathrsfs}
\usepackage{xypic}
\usepackage[
		bookmarks=true, bookmarksopen=true,%
    bookmarksdepth=3,bookmarksopenlevel=2,%
    colorlinks=true,%
    linkcolor=blue,%
    citecolor=blue]{hyperref}
\newtheorem{theorem}{Theorem}[section]
\newtheorem{lemma}[theorem]{Lemma}
\newtheorem{proposition}[theorem]{Proposition}
\newtheorem{corollary}[theorem]{Corollary}

\theoremstyle{definition}
\newtheorem{definition}[theorem]{Definition}

\newtheorem{notation}[theorem]{Notation}

\newtheorem{example}[theorem]{Example}

\newtheorem{remark}[theorem]{Remark}

\numberwithin{equation}{section}

\newcommand{\Real}{{\mathbb R}}
\newcommand{\Rational}{{\mathbb Q}}

\newcommand{\Integral}{{\mathbb Z}}

\title{
Barbell twists are natural}
     
\author[Yi Liu]{%
        Yi Liu} 
\address{%
				School of Mathematical Sciences/Beijing International Center for Mathematical Research, 
				Peking University\\
				Beijing 100871, China P.R.} 
\email{%
    liuyi@bicmr.pku.edu.cn}

\thanks{Partially supported by NSFC Grant 12525101, 
and National Key R\&D Program of China 2020YFA0712800}

\date{%
 \today} 

\begin{document}

\begin{abstract}
For any oriented smooth $4$--manifold $X$
diffeomorphic to $(S^2\times D^2)^{\natural n}$ ($n\geq0$),
the author establishes a natural isomorphism of abelian groups:
$$\mathrm{Mod}(X,\partial X)\cong \mathrm{Mod}(D^4,\partial D^4)\times\wedge^2H_2(X;\mathbb{Z}),$$
concerning the (smooth) boundary-fixing mapping class group of $X$.
For $n=2$, 
the Budney--Gabai barbell twist $\varphi\in\mathrm{Mod}(\mathcal{N},\partial\mathcal{N})$
is identified with a generator of the factor subgroup $\wedge^2H_2(\mathcal{N};\mathbb{Z})\cong\mathbb{Z}$.
Up to boundary-fixing diffeotopy, the barbell spines of $\mathcal{N}$
are completely classified by the bases of $H_2(\mathcal{N};\mathbb{Z})\cong\mathbb{Z}^2$,
forming a homogeneous set modeled on the group
$\mathrm{GL}(H_2(\mathcal{N};\mathbb{Z}))\cong\mathrm{GL}(2,\mathbb{Z})$.
Any barbell spine of $\mathcal{N}$ 
gives rise to an implanted barbell twist equal to $\varphi$ or $\varphi^{-1}$ 
in $\mathrm{Mod}(\mathcal{N},\partial \mathcal{N})$, 
according to the sign of the homological basis orientation.
\end{abstract}

\maketitle

\section{Introduction}
The \emph{barbell diffeomorphism}, as introduced by Budney and Gabai \cite{BG_knotted},
has played a particularly useful role
in recent $4$--dimensional differential topology.
The barbell diffeomorphism 
is concretely constructed up to boundary-fixing diffeotopy,
as an orientation-preserving self-diffeomorphism
of a model manifold diffeomorphic to $(D^2\times S^2)^{\natural 2}$.
Thinking of the manifold as two copies of $D^2\times S^2$
connected up with a $1$--handle $D^1\times D^3$, 
a spine of the manifold takes the shape
of two $2$--spheres connected by an arc, 
which looks like a barbell in the real world.

The barbell diffeomorphism
shows certain similarity with the usual Dehn twist of an annulus.
The (right-hand) Dehn twist along a simple closed curve,
as a boundary-fixing mapping class,
is unambiguously defined upon fixing an orientation of the ambient surface.
The purpose of this paper is to provide 
a model-free characterization of the barbell diffeomorphism
in a similar fashion.
In view of its (behavioral) similarity with the annulus Dehn twist,
we refer to the boundary-fixing mapping class 
of any Budney--Gabai barbell diffeomorphism
as a \emph{barbell twist}.

In this introduction,
we first state our main theorem for a more general family of $4$--manifolds,
which are diffeomorphic to $(D^2\times S^2)^{\natural n}$ (Theorem \ref{main_T}).
We derive three corollaries 
with easy proofs right after their statements 
(Corollaries \ref{main_corollary_transformation}, \ref{main_corollary_product}, and \ref{main_corollary_BG}).
Then we focus on the case with $n=2$,
and explain our equivalent reformulation of the barbell twist.
We discuss implanted barbell twists 
arising from different barbell spines,
as a guiding example
to a complete classification of barbell spines
(Example \ref{example_spine}).

\begin{notation}\label{Mod_notation}
Let $X$ be any orientable, connected, compact, smooth 
manifold with (possibly empty) boundary.
\begin{enumerate}
\item 
The orientation-preserving self-diffeomorphism group of $X$ 
is denoted as $\mathrm{Diff}^+(X)$.
The normal subgroup $\mathrm{Diff}^+(X,\partial X)$ 
of $\mathrm{Diff}^+(X)$ consists of the orientation-preserving self-diffeomorphisms
fixing $\partial X$ pointwise. 
The normal subgroup $\mathrm{Diff}_0(X,\partial X)$ of $\mathrm{Diff}^+(X)$ 
consists of the self-diffeomorphisms that are diffeotopic 
to the identity fixing $\partial X$ constantly.
\item
Denote the (smooth) \emph{boundary-fixing mapping class group} of $X$ as
$$\mathrm{Mod}(X,\partial X)=\mathrm{Diff}^+(X,\partial X)/\mathrm{Diff}_0(X,\partial X),$$
or simply $\mathrm{Mod}(X)$, if $\partial X$ is empty.
\item
For any orientable, connected, compact, smooth manifold $Y$ of the same dimension,
and for any smooth embedding $i\colon X\to Y$, 
we denote the \emph{implantation homomorphism} as  
$$i_*\colon \mathrm{Mod}(X,\partial X)\to \mathrm{Mod}(Y,\partial Y).$$
This refers to a well-defined, unique group homomorphism,
such that for any $g\in\mathrm{Diff}^+(X,\partial X)$
which fixes some neighborhood of $\partial X$ in $X$, 
$i_*([g])\in\mathrm{Mod}(Y,\partial Y)$ is represented by
$i_*(g)\in\mathrm{Diff}^+(Y,\partial Y)$
as obtained from $i\circ g\circ i^{-1}$ on $i(X)$
by extending with the identity on $Y\setminus i(X)$.
\end{enumerate}
\end{notation}

\begin{theorem}\label{main_T}
Let $X$ be an oriented $4$--manifold diffeomorphic to $(S^2\times D^2)^{\natural n}$,
for some integer $n\geq0$.
Then, there exists a unique group homomorphism
$$T\colon \wedge^2 H_2(X;\Integral) \to \mathrm{Mod}(X,\partial X),$$
such that the following properties all hold.
\begin{itemize}
\item
Let $Y$ be any oriented, connected, compact $4$--manifold with possibly empty boundary,
and $i\colon X\to Y$ be any orientation-preserving smooth embedding.
For all $\alpha,\beta\in H_2(X;\Integral)$,
and for all $\eta\in H_2(Y,\partial Y;\Integral)$,
the following formula holds in $H_2(Y,\partial Y;\Integral)$.
$$\left(i_*(T_{\alpha\wedge\beta})\right)_*(\eta)=\eta + I(i_*(\alpha),\eta)\,i_*(\beta)- I(i_*(\beta),\eta)\,i_*(\alpha).$$
\item
Let $i\colon X\to D^4$ be any smooth embedding, 
such that $D^4\setminus i(\mathrm{int}(X))$ is diffeomorphic to $D^4\#(D^3\times S^1)^{\natural n}$.
For any $\omega\in \wedge^2 H_2(X;\Integral)$, $i_*(T_\omega)$ is trivial in $\mathrm{Mod}(D^4,\partial D^4)$.
\item
Let $j\colon D^4\to X$ be any smooth embedding.
For any $\phi\in\mathrm{Mod}(X,\partial X)$,
there exist some unique $\psi\in\mathrm{Mod}(D^4,\partial D^4)$ and some unique $\omega\in\wedge^2H_2(X;\Integral)$,
such that the following commutative factorization holds in $\mathrm{Mod}(X,\partial X)$.
$$\phi=j_*(\psi)\cdot T_\omega.$$
\end{itemize}
Here, 
$\wedge^2H_2(Y;\Integral)\cong\Integral^{n(n-1)/2}$ 
is viewed as an additive free abelian group; the notation
$I\colon H_2(Y;\Integral)\times H_2(Y,\partial Y;\Integral)\to \Integral$
denotes the algebraic intersection number pairing. 
\end{theorem}

Theorem \ref{main_T} includes two trivial cases $D^4$ ($n=0$) and $D^2\times S^2$ ($n=1$).
The first listed property regarding the homomorphism $T$ may be called
the \emph{homological action formula for an implantation}, 
and the second called the \emph{triviality of unknotted implantations},
and the third called the \emph{factorization property}.
It is essentially the first two properties that determine $T$ uniquely,
and the third comes along as an extra feature.
Very little is known about the group $\mathrm{Mod}(D^4,\partial D^4)$.
For the proof of Theorem \ref{main_T}, see Section \ref{Sec-main_proof}.
See also Section \ref{Sec-misc} for two useful criteria for applications.

\begin{corollary}[the transformation formula]\label{main_corollary_transformation}
Assume $(X,T)$ as declared in Theorem \ref{main_T}, 
and $(X',T')$ a pair with the similar properties.
Suppose that $f\colon X\to X'$ is a diffeomorphism.
Then, for all $\alpha,\beta\in H_2(X;\Integral)$,
$$f_*\left(T_{\alpha\wedge\beta}\right)=
T'_{\pm f_*(\alpha)\wedge f_*(\beta)}.$$
Here, the plus or minus sign in the formula is determined 
according as $f$ is orientation preserving or reversing, respectively. 
\end{corollary}

\begin{proof}
Construct another homomorphism $\tilde{T}\colon\wedge^2 H_2(X;\Integral)\to \mathrm{Mod}(X,\partial X)$
by assigning 
$$\omega\mapsto f^{-1}_*\left(T'_{\pm f_*(\omega)}\right).$$
It is easy to check that $\tilde{T}$ satisfies the same properties
as declared for $T$, with respect to $X$.
So, $\tilde{T}$ agrees with $T$ 
by the uniqueness of $T$ (Theorem \ref{main_T}).
The asserted formula follows immediately.
\end{proof}

The following consequence generalizes \cite[Theorem 5.7]{BG_knotted} ($n=2$) 
to all integers $n\geq0$,
and strengthens its conclusion with the naturality.

\begin{corollary}[{natural splitting}]\label{main_corollary_product}
For any oriented smooth $4$--manifold $X$ diffeomorphic to $(S^2\times D^2)^{\natural n}$,
there exists a natural isomorphism of abelian groups as follows.
$$\mathrm{Mod}(X,\partial X)\cong\mathrm{Mod}(D^4,\partial D^4)\times\wedge^2H_2(X;\Integral).$$
Here, the naturality means that for any orientation-preserving diffeomorphism
$X\to X'$, the above isomorphism commutes with the induced isomorphisms
on both sides.
\end{corollary}

\begin{proof}
A direct product decomposition is determined by 
the unique commutative factorization property in Theorem \ref{main_T}.
Note that the implantation homomorphism $j_*$
does not depend on the choice of any orientation-preserving smooth embedding $j\colon D^4\to X$.
The naturality follows from the transformation formula (Corollary \ref{main_corollary_transformation}.
\end{proof}

\begin{remark}\label{main_corollary_product_remark}
Without naturality,
the existence of a splitting as in Corollary \ref{main_corollary_product}
is known, due to Krushkal, Mukherjee, Powell, and Warren \cite[Theorem 1.8]{KMPW_cork}.
See also Ren--Sullivan--Wedrich--Wills--Zhang \cite[Theorem 1.3]{RSWWZ_lasagna}
for a more general theorem regarding the boundary-fixing mapping class group
of connected sums of boundary-connect sums of $D^2\times S^2$.
For any oriented topological $4$--manifold $X$ homeomorphic to $(S^2\times D^2)^{\natural n}$,
the topological boundary-fixing mapping class group $\pi_0(\mathrm{Homeo}^+(X,\partial X))$
is naturally isomorphic to $\wedge^2 H_2(X;\Integral)$.
In fact, Orson and Powell determine the topological boundary-fixing mapping class group
for all simply connected, compact, topological $4$--manifolds,
implying the above $X$ as a special case \cite[Theorem A]{OP_mcg}.
\end{remark}

Informally, a \emph{barbell} $\Gamma=P_a\cup I\cup P_b$ comprises 
an ordered pair of oriented $2$--spheres, called the \emph{head bell} $P_a$ and the \emph{tail bell} $P_b$,
and an arc, called the \emph{bar} $I$.
A \emph{barbell spine} $\Gamma$ sits in a $4$--manifold $\mathcal{N}$, 
suggesting a boundary-connect sum decomposition of $\mathcal{N}$ as
$(D^2\times P_a)\natural(D^2\times P_b)$ along a cut $3$--disk dual to $I$.
See Definitions \ref{barbell_def} and \ref{barbell_spine_def} for details.
For a review of the Budney--Gabai model, see Section \ref{Sec-BG}.

\begin{corollary}[identification with the Budney--Gabai model]\label{main_corollary_BG}
	Let $\mathcal{N}$ be a model thickened barbell
	with the model barbell spine $\Gamma=P_a\cup I\cup P_b$.
	The barbell twist arising from the Budney--Gabai model
	is recognized as
	$$\varphi=T_{[P_a]\wedge[P_b]}$$
	in $\mathrm{Mod}(\mathcal{N},\partial\mathcal{N})$.
	See Definition \ref{BG_barbell_twist_def}.
\end{corollary}

\begin{proof}
	Since $\mathcal{N}$ is diffeomorphic to $(S^2\times D^2)^{\natural 2}$,
	the ordered pair of oriented $2$--spheres $(P_a,P_b)$ determines
	a generator $[P_a]\wedge[P_b]$ of $\wedge^2H_2(\mathcal{N};\Integral)\cong\Integral$.
	
	According to a description of the Budney--Gabai model (see Theorem \ref{BG_construction}),
	the implantation of $\varphi$ in $\mathrm{Mod}(D^4,\partial D^4)$ is trivial
	for some smooth embedding $\mathcal{N}\to D^4$ with complement closure 
	diffeomorphic to $D^4\#(D^3\times S^1)^{\natural 2}$
	(see Corollary \ref{BG_construction_corollary}).
	The same implantation also 
	annihilates the $\wedge^2H_2(M;\Integral)$ factor of $\varphi$
	in its unique commutative factorization (Theorem \ref{main_T}).
	It follows that the $\mathrm{Mod}(D^4,\partial D^4)$ factor of $\varphi$ must be trivial.	
	We infer
	$$\varphi=T_{m\cdot[P_a]\wedge[P_b]},$$
	for some $m\in\Integral$.
	
	With respect to the canonical inclusion $\mathcal{N}\to W_{\mathcal{N}}$
	of $\mathcal{N}$ into the oriented doubling 
	$W_{\mathcal{N}}=\mathcal{N}\cup_{\partial\mathcal{N}}(-\mathcal{N})$
	(diffeomorphic to $(S^2\times S^2)^{\natural 2}$),
	the homological action formula for the implantation of $T_{[P_a]\wedge[P_b]}$ (Theorem \ref{main_T})
	agrees with 
	that of the implantation of $\varphi$ on $H_2(W_{\mathcal{N}};\Integral)$ (see Corollary \ref{BG_construction_corollary}).
	This implies $m=1$. Therefore, $\varphi$ agrees with $T_{[P_a]\wedge[P_b]}$
	in $\mathrm{Mod}(\mathcal{N},\partial\mathcal{N})$.
\end{proof}

In view of Corollary \ref{main_corollary_BG},
for	any oriented $4$--manifold $\mathcal{N}$ diffeomorphic to $(S^2\times D^2)^{\natural 2}$,
we can refer to the boundary fixing mapping class
\begin{equation}\label{T_omega_notation}
T_\omega\in\mathrm{Mod}(\mathcal{N},\partial \mathcal{N})
\end{equation}
as \emph{the} barbell twist,
upon fixing a generator $\omega$ of $\wedge^2H_2(\mathcal{N};\Integral)\cong\Integral$.

A practical advantage of the reformulation (\ref{T_omega_notation}) is that it allows us
to declare a barbell twist purely in terms of 
its behavior under implantations
(compare Theorem \ref{BG_construction} and Definition \ref{BG_barbell_twist_def}).

\begin{remark}[dependence on the orientation]\label{main_corollary_BG_remark}
	If we flip the orientation of $\mathcal{N}$ and 
	simultaneously the sign of $\omega$,
	the resulting barbell twist does not change
	in $\mathrm{Mod}(\mathcal{N},\partial \mathcal{N})$ (Corollary \ref{main_corollary_transformation}).
	Therefore, without designating an orientation of $\mathcal{N}$,
	we could speak of the barbell twist of $\mathcal{N}$,
	upon fixing a generator of 
	$H_4(\mathcal{N},\partial \mathcal{N};\Integral)\otimes\wedge^2H_2(\mathcal{N};\Integral)\cong\Integral$.
	This makes the barbell twist a bit more like
	the ($2$--dimensional) annulus twist,
	which is naturally defined in $\mathrm{Mod}(\mathcal{A},\partial \mathcal{A})$ for an annulus $\mathcal{A}$,
	upon fixing a generator of $H_2(\mathcal{A},\partial \mathcal{A};\Integral)\cong\Integral$.
\end{remark}

For any (legally embedded) barbell $\Gamma=P_a\cup I\cup P_b$ in an oriented, connected, compact
$4$--manifold $X$ (Definition \ref{barbell_def}), 
we can take 
any (unparametrized, compatibly oriented) smooth $4$--manifold neighborhood $\mathcal{N}\subset X$
of $\Gamma$, diffeomorphic to $(S^2\times D^2)^{\natural 2}$ 
and containing $\Gamma$ as a barbell spine (Definition \ref{barbell_spine_def}).
The $\partial X$--fixing (and $\Gamma$--fixing) diffeotopy class of 
$\mathcal{N}$ depends only on $\Gamma$.
This gives rise to a well-defined \emph{implanted barbell twist}
\begin{equation}\label{T_Gamma_notation}
T_\Gamma\in \mathrm{Mod}(X,\partial X),
\end{equation}
depending only on $(X,\Gamma)$,
as the implantation of $T_{[P_a]\wedge[P_b]}\in\mathrm{Mod}(\mathcal{N},\partial \mathcal{N})$
via the inclusion $\mathcal{N}\to X$.
The reformulation (\ref{T_Gamma_notation}) agrees with the similar notion introduced by Budney and Gabai 
in terms of a particular model \cite[Definition 5.11]{BG_knotted}.

An implanted barbell twist tends to be more invariant than the barbell itself,
as the following example illustrates.

\begin{example}[different barbell spines]\label{example_spine}
	Let $\mathcal{N}$ 
	be an oriented $4$--manifold diffeomorphic to $(S^2\times D^2)^{\natural 2}$.
	
	Pick an ordered pair of smooth embedded, mutually disjoint, oriented $2$--spheres 
	$$S_a,S_b\subset\partial\mathcal{N},$$
	such that $[S_a],[S_b]$ form a basis 
	of the free abelian group 
	$H_2(\partial\mathcal{N};\Integral)\cong H_2(\mathcal{N};\Integral)\cong\Integral^2$.
	In this case, 
	$S_a\sqcup S_b$ necessarily has connected complement in $\partial\mathcal{N}$.
	Pick a smoothly embedded arc 
	$$J\subset\partial \mathcal{N},$$
	connecting $S_a$ and $S_b$,
	such that $J$ intersects $S_a\sqcup S_b$ transversely and only at the endpoints $\partial J$.
	Obtain a parallel copy $P_a\cup I\cup P_b$ of the barbell $S_a\cup J\cup S_b$ 
	in a parallel copy of $\partial\mathcal{N}$ (in a collar neighborhood in $\mathcal{N}$),
	denoted as
	$$\Gamma\subset\mathrm{int}(\mathcal{N}).$$
	
	It can be shown that 
	any $\Gamma$ arising from the above construction forms a barbell spine of $\mathcal{N}$
	(see Corollary \ref{characterization_spine}).
	Its implanted barbell twist agrees with the barbell twist 
	of $\mathcal{N}$ with respect to the generator $[S_a]\wedge[S_b]$ of $\wedge^2H_2(\mathcal{N};\Integral)\cong\Integral$.
	Namely,
	$$T_\Gamma=T_{[S_a]\wedge[S_b]}$$
	holds in $\mathrm{Mod}(\mathcal{N},\partial\mathcal{N})$,
	according to our model-free reformulations
	(\ref{T_omega_notation}) and (\ref{T_Gamma_notation}).
	In particular, $T_\Gamma$ depends only on $[S_a]\wedge[S_b]$.
	
	On the other hand, 
	for different choices of the initial objects $(S_a,S_b,J)$,
	in general,	the resulting barbell spines 
	$\Gamma=P_a\cup I\cup P_b$ are different
	up to boundary-fixing diffeotopy of $\mathcal{N}$.
	For example, the basis $[S_a],[S_b]$ of $H_2(\mathcal{N};\Integral)$
	serves as an obvious homological invariant 
	for distinguishing some different pairs.
\end{example}
	
In Section \ref{Sec-classification_spine}, 
we provide a complete classification of barbell spines
up to boundary-fixing diffeotopy (Theorem \ref{classification_spine}). 
In terms of Example \ref{example_spine},
the classification says that barbell spines of $\mathcal{N}$ 
all arise as constructed in Example \ref{example_spine},
and their boundary-fixing diffeotopy classes correspond bijectively
to the bases of $H_2(\mathcal{N};\Integral)\cong\Integral^2$,
which form a homogeneous set modeled on $\mathrm{GL}(H_2(\mathcal{N};\Integral))\cong\mathrm{GL}(2,\Integral)$ 
(see Corollaries \ref{classification_spine_GL} and \ref{characterization_spine}).

To conclude,
between differential topological objects associated to
an oriented smooth $4$--manifold $\mathcal{N}$ diffeomorphic to $(S^2\times D^2)^{\natural 2}$,
and linear objects associated to $H_2(\mathcal{N};\Integral)$ isomorphic to $\Integral^2$,
there is a natural and classifying analogy, namely,
$$\mbox{barbell spine}:\mbox{barbell twist}::\mbox{homological basis}:\mbox{homological orientation}.$$

\subsection*{Ingredients}
Our proof of the main theorem (Theorem \ref{main_T}) involves three ingredients, as listed below.

\begin{itemize}
\item \emph{The extended homological action}.
For any compact $4$--manifold $X$,
the usual homological action of boundary-fixing homeomorphisms
on $H_2(X;\Integral)$ admits a natural refinement,
as a group homomorphism 
$\mathrm{Homeo}(X,\partial X)\to 
\mathrm{Hom}\left(H_2(X,\partial X;\Integral),H_2(X;\Integral)\right)\rtimes\mathrm{Aut}(H_2(X;\Integral))$.
If $X$ is oriented, smooth, and diffeomorphic to $(D^2\times S^2)^{\natural n}$,
the homomorphism simplifies to be
$$\mathrm{Mod}(X,\partial X)\to \wedge^2H_2(X;\Integral).$$
We revisit relevant facts, and derive a homological action formula for implantations
in terms of the simplified homomorphism (Corollary \ref{implanted_homological}).
The formula is used in the proof of Theorem \ref{main_T}.
See Section \ref{Sec-eha}.
\item \emph{Extendability for boundary diffeomorphisms of a $1$--handlebody}.
For any smooth $4$--manifold $Z$ diffeomorphic to $(S^1\times D^3)^{\natural n}$,
we prove that any orientation-preserving self-diffeomorphism of $\partial Z$
admits an extension as an orientation-preserving self-diffeomorphism of $Z$
(Theorem \ref{extendability_1hdlbdy}).
See Section \ref{Sec-1hdlbdy}.
\item \emph{The light bulb theorem for several disks}.
Generalizing the (single) disk version
of the light bulb theorem due to Kosanovi\'c and Teichner \cite{KT_lbt},
we establish the several disks version (Theorem \ref{lbt_disks}).
See Section \ref{Sec-2hdls}.
\end{itemize}

Results from the above ingredients have been formulated in greater generality
than it is needed for proving Theorem \ref{main_T}.
To develop them as a toolbox is a hidden motivation of this paper.

We outline our proof of Theorem \ref{main_T} as follows.
Let $X$ be an oriented smooth $4$--manifold diffeomorphic to $(S^2\times D^2)^{\natural n}$.
The essential work is to construct a canonical group isomorphism
$$\mathrm{Mod}(X,\partial X)\to\mathrm{Mod}(D^4,\partial D^2)\times \wedge^2H_2(X;\Integral),$$
whose inverse restricted to the $\wedge^2H_2(X;\Integral)$ will be the asserted group homomorphism $T$.
The construction of a group homomorphism of the above form
involves an arbitrarily chosen unknotted smooth embedding $X\to D^4$,
and makes use of the first ingredient (the extended homological action).
The surjectivity is proved by implanting a collection of model barbell twists.
The injectivity is proved by using the third ingredient (the light bulb theorem for several disks).
The second ingredient (extendability for boundary diffeomorphisms of a $1$--handlebody) 
accounts for the independence of the construction on the choice of $X\to D^4$.
See Section \ref{Sec-main_proof} for details.

Our classification of barbell spines is proved without using the above ingredients.
We refer the reader directly to Section \ref{Sec-classification_spine}.
The exposition of Section \ref{Sec-classification_spine} relies only on 
Definitions \ref{barbell_def} and \ref{barbell_spine_def} 
in Section \ref{Sec-BG}.

\subsection*{Terminology}
For a smooth map $f\colon X\to Y$ 
between compact, smooth manifolds with boundary,
we say that $f$ is \emph{neat} if $\partial X$ is equal to $f^{-1}(\partial Y)$,
and if $f$ is transverse to $\partial Y$ in $Y$.
This term appears in \cite{KT_lbt}.
In the literature, a neat map is also called a ``proper'' smooth map,
in conflict with a similar term in general topology, 
which means ``every compact set has compact preimage'' therein.
A homotopy or a smooth isotopy \emph{relative to} a subspace
is assumed to fix the subspace constantly, as usual in the literature.
A \emph{diffeotopy} is synonymous to 
an ``ambient smooth isotopy''.
For other standard terminology in differential topology,
we refer to Wall's textbook \cite{Wall_book}.

\subsection*{Organization}
In Section \ref{Sec-BG}, we review the Budney--Gabai model of a barbell twist.
In Section \ref{Sec-eha}, we discuss the extended homological action, 
and derive a homological action formula for implantations (Corollary \ref{implanted_homological}).
In Section \ref{Sec-1hdlbdy}, we prove 
the extendability for boundary diffeomorphisms of a $1$--handlebody (Theorem \ref{extendability_1hdlbdy}).
In Section \ref{Sec-2hdls}, we prove the light bulb theorem for several disks (Theorem \ref{lbt_disks}),
and derive a simplified criterion regarding embedding uniqueness of several $2$--handles (Corollary \ref{uniqueness_bnb_2hdls}.
In Section \ref{Sec-main_proof}, we prove the main theorem (Theorem \ref{main_T}).
In Section \ref{Sec-misc}, we mention two criteria (Propositions \ref{criterion_T} and \ref{transformation_T_general}) 
for facilitating applications of the main theorem.
In Section \ref{Sec-classification_spine},
we characterize and classify barbell spines up to boundary-fixing diffeotopy,
in any smooth $4$--manifold diffeomorphic to $(S^2\times D^2)^{\natural 2}$.

\subsection*{Acknowledgment}
The author is grateful to Mark Powell and Qiuyu Ren for valuable comments,
and to Jianfeng Lin for helpful communication.

\section{The Budney--Gabai model of a barbell twist}\label{Sec-BG}
In this section, we review the Budney--Gabai model of a barbell twist,
based on \cite[Section 5]{BG_knotted}.
A barbell twist described this way is called 
a \emph{model barbell twist} (Definition \ref{BG_barbell_twist_def}).
We also define the terms 
\emph{legally embedded barbell} (Definition \ref{barbell_def}) 
and \emph{barbell spine} (Definition \ref{barbell_spine_def})
for our discussion.

\begin{definition}\label{barbell_def}
	For any smooth $4$--manifold $X$,
	a \emph{legally embedded barbell} (or simply, a \emph{barbell}) 
	in $X$ refers to a triple 
	$$(P_a,P_b,I),$$ as follows.
	The items $P_a,P_b\subset X$ are smoothly embedded, mutually disjoint, oriented $2$--spheres
	with trivializable normal vector bundles in $X$; the item $I\subset X$ is a smoothly embedded arc
	connecting $P_a$ and $P_b$, and intersects $P_a$ and $P_b$ 
	only in $\partial I$ and along normal directions.
	We call $P_a$ the \emph{head bell}, and $P_b$ the \emph{tail bell},
	and $I$ the \emph{bar}.
	We often adopt an informal notation 
	$$\Gamma=P_a\cup I\cup P_b,$$
	keeping in mind that the bells are ordered and oriented.
\end{definition}

\begin{definition}\label{barbell_spine_def}
	For any smooth $4$--manifold $\mathcal{N}$ is diffeomorphic to  $(S^2\times D^2)^{\natural2}$,
	a legally embedded barbell $\Gamma=P_a\cup I\cup P_b$ in $\mathrm{int}(\mathcal{N})$ is said to 
	form a \emph{barbell spine} of $\mathcal{N}$,
	if there exists some neatly embeddded $3$--disk in $\mathcal{N}$,
	intersecting $\Gamma$	transversely at a unique point in $\mathrm{int}(I)$, 
	and witnessing a boundary-connect sum decomposition
	of $\mathcal{N}$ into summands of the form $P_a\times D^2$ and $P_b\times D^2$.
	With respect to a barbell spine $\Gamma$, 
	we say that $\mathcal{N}$ is a \emph{thickened barbell}.
\end{definition}

\begin{theorem}[Budney--Gabai]\label{BG_construction}
	There exists some tuple 
	$$\left(\mathcal{N},\Gamma,\varphi\right)$$ 
	as follows.
	\begin{itemize}
	\item 
	In the tuple, the item 
	$$\mathcal{N}\subset \Real^4$$ is 
	a smoothly embedded, oriented $4$--manifold diffeomorphic to $(S^2\times D^2)^{\natural 2}$;
	the item 
	$$\Gamma=P_a\cup I\cup P_b$$
	is a barbell spine of $\mathcal{N}$;
	the item 
	$$\varphi\in\mathrm{Mod}(\mathcal{N},\partial \mathcal{N})$$ 
	is a boundary-fixing mapping class;
	see Definitions \ref{barbell_def} and \ref{barbell_spine_def},
	and Notation \ref{Mod_notation}.
	Moreover, the tuple satisfies the following declared properties.
	\item
	There exist some neatly embedded, mutually disjoint $3$--disks 
	$$(Q_a,\partial Q_a),(Q_b,\partial Q_b)\subset(\Real^4\setminus \mathrm{int}(\mathcal{N}),\partial \mathcal{N})$$
	and some tubular neighborhoods of the form
	$Q_a\times[-1,1]$ and $Q_b\times[-1,1]$
	in $\Real^4\setminus\mathrm{int}(\mathcal{N})$,
	such that the region 
	$$V=\mathcal{N}\cup (Q_a\times [-1,1])\cup (Q_b\times [-1,1])$$ 
	in $\Real^4$ is diffeomorphic to $D^4$ after corner smoothing 
	along $\partial Q_a\times\{-1,1\}$ and $\partial Q_b\times\{-1,1\}$,
	and	such that $\partial Q_a\sqcup\partial Q_b$ 
	is smoothly isotopic to $P_a\sqcup P_b$ in $\mathcal{N}$, 
	respecting the labeling of components.
	\item
	There exist some neatly embedded, mutually disjoint, 
	oriented $2$--disks 
	$$(E_a,\partial E_a),(E_b,\partial E_b)\subset (\mathcal{N},\partial \mathcal{N}),$$ 
	such that $\partial E_{a/b}$
	intersects $P_a\cup P_b$ transversely 
	at a unique, positive intersection point in $P_{a/b}$.
	Here, the subscript $a/b$ means $a$ or $b$ respectively.
	\item
	There exists some 
	$$h\in\mathrm{Diff}_0(V,\partial V),$$
	such that $h$ 
	fixes some neighborhood of $\partial \mathcal{N}\cup(Q_a\times[-1,1])\cup(Q_b\times[-1,1])$,
	and such that 
	$$g=h|_{\mathcal{N}}$$
	is a representative of $\varphi$ in $\mathrm{Diff}^+(\mathcal{N},\partial \mathcal{N})$.
	Moreover, the neatly embedded, oriented $2$--disk $g(E_a)$ in $\mathcal{N}$
	takes the form as an oriented tubing connect sum in $\mathcal{N}$
	of $E_a$ with a parallel nearby copy of $P_b$,
	up to smooth isotopy in $\mathcal{N}$ relative to $\partial E_a$;
	similarly, $g(E_b)$ is smoothly isotopic to $E_b$ tubed 
	with an orientation-reversed copy of $P_a$,
	relative to $\partial E_b$.
	\end{itemize}
\end{theorem}

\begin{remark}\label{BG_construction_remark}
	Theorem \ref{BG_construction} is a descriptive summary 
	of Budney and Gabai's model in \cite[Definition 5.1 and Construction 5.3]{BG_knotted}.
	To describe their construction informally, 
	we can drag a middle portion of the $1$--handle $[-1,1]\times Q_a$
	in the $4$--disk $V$,
	sweep out the $2$--sphere $P_b$ therein,
	and move on to the original position,
	just like rope skipping.
	This makes a film of a boundary-fixing diffeotopy of $V$
	ending up with some $h\in\mathrm{Diff}_0(V,\partial V)$,
	which yields a representative $g=h|_{\mathcal{N}}$ of $\varphi$
	as declared in Theorem \ref{BG_construction}.
	Our objects $(V,Q_a\times[-1,1],Q_b\times[-1,1])$
	correspond to Budney and Gabai's objects $(W,N(\alpha_1),N(\alpha_2))$
	in \cite[Definition 5.1 and Construction 5.3]{BG_knotted},
	and $(P_a,P_b,E_a,E_b)$ correpsond to their $(P_1,P_2,E_1,E_2)$.
	Our objects $(X,g)$ correspond to their objects $(\mathcal{NB},\beta)$,
	which are called the \emph{thick barbell} and 
	the \emph{barbell diffeomorphism} therein.	
	To make it clear, both our $g$ and their $\beta$ 
	are only specified up to relative diffeotopy
	(rather than, say, up to conjugacy of
	boundary-fixing diffeomorphisms that are relative diffeotopically trivial).
	See Section \cite[Section 5]{BG_knotted}
	for details and for further discussion.
\end{remark}

\begin{corollary}\label{BG_construction_corollary}
	Assume $(\mathcal{N},P_a,P_b,I,\varphi)$ as declared in Theorem \ref{BG_construction}.
	\begin{enumerate}
	\item There exists some embedding $i\colon\mathcal{N}\to D^4$ 
	with $D^4\setminus i(\mathrm{int}(\mathcal{N}))$ diffeomorphic to $D^4\# (D^3\times S^1)^{\natural 2}$, 
	such that $i_*(\varphi)$ is trivial in $\mathrm{Mod}(D^4,\partial D^4)$.
	\item Let $W_{\mathcal{N}}=\mathcal{N}\cup_{\partial \mathcal{N}}(-\mathcal{N})$ be the oriented doubling of $\mathcal{N}$ along $\partial \mathcal{N}$,
	and $i\colon\mathcal{N}\to W_{\mathcal{N}}$ be the canonical inclusion.
	For any $\eta\in H_2(W_{\mathcal{N}};\Integral)$,
	$$(i_*(\varphi))(\eta)=\eta+ I(i_*([P_a]),\eta)\, i_*([P_b]) -I(i_*([P_b]),\eta)\, i_*([P_a]).$$
	\end{enumerate}
\end{corollary}

\begin{proof}
	With the notations in Theorem \ref{BG_construction},
	the first statement follows by
	taking $i\colon \mathcal{N}\to D^4$ to be 
	the inclusion of $\mathcal{N}$ into 
	slightly larger neighborhood of the smooth $4$--disk $V$ in $\Real^4$.
	Note that $i_*(\varphi)\in\mathrm{Mod}(D^4,\partial D^4)$ is trivial
	because $h$ already lies in $\mathrm{Diff}_0(V,\partial V)$.
	The second statement follows by checking the formula
	on the basis $i_*([P_a]),i_*([P_b]),[\hat{E}_a],[\hat{E}_b]$
	of $H_2(W_{\mathcal{N}};\Integral)\cong\Integral^4$,
	where $\hat{E}_{a/b}\subset W_{\mathcal{N}}$ denotes
	the $2$--sphere obtained by doubling $E_{a/b}$ along $\partial E_{a/b}$.
	Note that $W_{\mathcal{N}}$ is diffeomorphic to $(S^2\times S^2)^{\# 2}$.	
\end{proof}

\begin{definition}\label{BG_barbell_twist_def}
	Let $(\mathcal{N},\Gamma,\varphi)=(\mathcal{N},P_a,P_b,I,\varphi)$ 
	be any tuple as declared in Theorem \ref{BG_construction},
	which we refer to as a \emph{model tuple}.
	We call the boundary-fixing mapping class 
	$$\varphi\in\mathrm{Mod}(\mathcal{N},\partial \mathcal{N})$$ 
	the \emph{model barbell twist} (or the \emph{Budney--Gabai twist}) of 
	the \emph{model thickened barbell} $\mathcal{N}$,
	with the respect to the \emph{model barbell spine} $\Gamma$.
\end{definition}

\begin{remark}\label{BG_barbell_twist_def_remark}
	For any model tuples $(\mathcal{N},\Gamma,\varphi)$ and $(\mathcal{N}',\Gamma',\varphi')$,
	it is easy to construct an orientation-preserving diffeomorphism $f\colon\mathcal{N}\to \mathcal{N}'$,
	transforming $\Gamma=(P_a,P_b,I)$ into $\Gamma'=(P'_a,P'_b,I')$.
	If we assume Theorem \ref{main_T},
	we will infer $f_*(\varphi)=\varphi'$ 
	from Corollaries \ref{main_corollary_transformation} and \ref{main_corollary_BG}.
	Therefore, all model tuples will turn out to be equivalent up to orientation-preserving diffeomorphisms,
	and in this sense, the same as 
	the particular one constructed by Budney and Gabai \cite[Construction 4.5]{BG_knotted}.
\end{remark}

\section{The extended homological action}\label{Sec-eha}
In this section, we establish
a general (middle-dimensional, integral) homological action formula 
for implantations (Corollary \ref{implanted_homological}).
For any oriented, connected, compact, smooth $2m$--manifold $X$
with $H_m(X;\Integral)$ torsion-free and 
with $H_m(X;\Integral)\to H_m(X,\partial X;\Integral)$ of finite image,
the formula computes the induced linear automorphic action
of $i_*(f)$ on $H_m(Y,\partial Y;\Integral)$ for any $[f]\in\mathrm{Mod}(X,\partial X)$
implanted via $i\colon X\to Y$.
This provides a unified treatment for 
the annulus Dehn twist (Example \ref{Dehn_example})
and the Budney--Gabai barbell twist (Example \ref{BG_Delta_f_example}).

Our formula involves a term $\Delta f$ called the variation of $f$,
which becomes a tensor in $H_m(X;\Integral)\otimes H_m(X;\Integral)$
under our assumption.
This tensor is symmetric or alternating, depending on the parity of $m$.
We refer to \cite{DK_periodicity,Kauffman_periodicity,Lamotke,Saeki_stable_mcg}
for some original developments of the variation,
and to \cite[Section 2.2]{OP_mcg} for a quick review.
In what follows, we supply a self-contained treatment of relevant facts, 
for the reader's convenience, and derive our formula.

For any continuous self-map $f\colon X\to X$
of a topological space $X$, and for any subspace $A\subset X$ fixed by $f$,
there is a well-defined abelian group homomorphism
\begin{equation}\label{Delta_f_notation}
\Delta f\colon H_k(X,A)\to H_k(X),
\end{equation}
called the \emph{variation} of $f$ with respect to $(X,A)$, as follows.

We allow $H_k$ to be the (singular) homology for any dimension $k$,
allowing any abelian group coefficients.
First obtain the homomorphism 
$f_\sharp-\mathrm{id}\colon C_k(X)\to C_k(X)$
of the the singular $k$--chain group $C_k(X)$.
This descends to a homomorphism $C_k(X,A)\to C_k(X)$
of the singular relative $k$--chain group $C_k(X,A)=C_k(X)/C_k(A)$,
as $f$ fixes $A$. 
The homomorphism $C_k(X,A)\to C_k(X)$ commutes obviously 
with the boundary operators of the complexes $C_*(X,A)$ and $C_*(X)$,
inducing a well-defined homomorphism $\colon H_k(X,A)\to H_k(X)$,
denoted as $\Delta f$.

From the above construction, the variation of $f$ takes the form
\begin{equation}\label{Delta_f_def}
\Delta f\colon [z]\mapsto [f_{\sharp}(z)-z],
\end{equation}
representing any $[z]\in H_k(X,A)$ with some $z\in C_k(X)$ with $\partial z\in C_{k-1}(A)$.

\begin{proposition}\label{variation_properties}
	Assume $(X,A,f)$ as in (\ref{Delta_f_notation}).
	\begin{enumerate}
	\item
	The homomorphism $\Delta f\colon H_k(X,A)\to H_k(X)$ depends only on
	the homotopy class of $f\colon X\to X$ relative to $A$.
	\item 
	The following composite homomorphism is equal to $f_*-\mathrm{id}$.
	$$\xymatrix{
	H_k(X) \ar[r]^-{\mathrm{incl}_*} & H_k(X,A) \ar[r]^-{\Delta f} & H_k(X)
	}$$
	\item 
	If $(X',A')\subset (X,A)$ is a subpair with $f(X')\subset X'$, then
	the following diagram of abelian group homomorphisms commutes.
	$$\xymatrix{
	H_k(X',A') \ar[r]^-{\Delta f'} \ar[d]_{\mathrm{incl_*}} & H_k(X') \ar[d]^{\mathrm{incl}_*} \\
	H_k(X,A) \ar[r]^-{\Delta f} & H_k(X)
	}$$
	Here, $f'\colon X'\to X'$ denotes the restricted continuous self-map of $X'$ fixing $A'$.
	\end{enumerate}	
\end{proposition}

\begin{proof}
	Obvious from the defining construction (\ref{Delta_f_def}).
\end{proof}

\begin{corollary}[the extended homological action]\label{variation_properties_corollary}
	For any topological space pair $(X,A)$, 
	the assignment 
	$$f\mapsto(\Delta f,f_*)$$
	determines a group homomorphism  
	$$\mathrm{Homeo}(X,A)\to \mathrm{Hom}(H_k(X,A),H_k(X))\rtimes \mathrm{Aut}(H_k(X)),$$
	for every dimension $k$.
	Here, $\mathrm{Homeo}(X,A)$ denotes 
	the group 
	consisting of all self-homeomorphisms $X\to X$ which fix $A$;
	$\mathrm{Aut}(H_k(X))$ denotes the automorphism group of the abelian group $H_k(X)$;
	the multiplication rule for the semidirect product of groups
	is $(a,\phi)\,(b,\psi)=(a+\phi\cdot b,\,\phi\psi)$ by convention.
\end{corollary}

\begin{proof}
	For any $f,g\in\mathrm{Homeo}(X,A)$,
	we observe the following chain-level identity 
	$$(fg)_\sharp(u)-u=(f_\sharp(u)-u)+f_\sharp\left(g_\sharp(u)-u\right),$$
	for any $u\in C_k(X;\Integral)$.
	This implies the following identity 
	of abelian group homomorphisms $H_k(X,A)\to H_k(X)$:
	\begin{equation}\label{c_f_crosshom}
	\Delta(fg)=\Delta f+f_*\circ \Delta g,
	\end{equation}
	by (\ref{Delta_f_def}),
	where $f_*$ 	denotes the induced group automorphism $H_k(X)\to H_k(X)$.
	Writing in another way, we obtain the relation
	$$\left(\Delta(fg),(fg)_*\right)=\left(\Delta f+f_*\cdot \Delta g,\,f_*g_*\right)$$
	for any $f,g\in\mathrm{Homeo}(X,A)$,
	proving that the assignment $f\mapsto (\Delta f,f_*)$ defines a group homomorphism
	of $\mathrm{Homeo}(X,A)$ to 
	$\mathrm{Hom}(H_k(X,\partial X),H_k(X)\rtimes \mathrm{Aut}(H_k(X))$,
	as asserted.
\end{proof}

\begin{proposition}\label{variation_tensor}
	Let $X$ be an oriented, connected, compact smooth $2m$--manifold
	with possibly empty boundary,	for some integer $m\geq0$.
	Suppose that $H_m(X;\Integral)$ is torsion-free,
	and the natural homomorphism $H_m(X;\Integral)\to H_m(X,\partial X;\Integral)$
	has finite image.
	
	Then, 
	the natural linear homomorphism 
	$$H_m(X;\Integral)\otimes H_m(X;\Integral)\to 
	\mathrm{Hom}(H_m(X,\partial X;\Integral),H_m(X;\Integral))$$
	as defined by the algebraic intersection number pairing
	of the first factor $H_m(X;\Integral)$ with $H_m(X,\partial X;\Integral)$ 
	is an isomorphism,
	and	$\mathrm{Homeo}(X,\partial X)$ 
	induces the trivial linear action
	on $H_m(X;\Integral))$.		
	In particular, the assignment
	$$[f]\mapsto \Delta f$$
	determines a well-defined group homomorphism
	$$\mathrm{Mod}(X,\partial X)\to H_m(X;\Integral)\otimes H_m(X;\Integral),$$
	viewing $\Delta f$ as living in $H_m(X;\Integral)\otimes H_m(X;\Integral)$.
	
	Moreover, for any $[f]\in\mathrm{Mod}(X,\partial X)$, 
	the tensor
	$\Delta f\in H_m(X;\Integral)\otimes H_m(X;\Integral)$
	is alternating if $m$ is even, or symmetric if $m$ is odd.
\end{proposition}

\begin{proof}
	The assumption that $H_m(X;\Integral)\to H_m(X,\partial X;\Integral)$
	has finite image implies that $H_m(\partial X;\Integral)\to H_m(X;\Integral)$
	has finite cokernel.
	For any $f\in\mathrm{Homeo}(X,\partial X)$, 
	we infer that $f_*\in\mathrm{GL}(H_m(X;\Integral))$ is the identity,
	because $f$ fixes $\partial X$ and $H_m(X;\Integral)$ 
	is assumed to be torsion-free.
	The torsion-free assumption also guarantees
	that $H_m(X;\Integral)\otimes H_m(X;\Integral)$ is naturally isomorphic to
	$\mathrm{Hom}(H_m(X,\partial X;\Integral),H_m(X;\Integral))$,
	by the Lefschetz--Poincar\'e duality.
	
	In particular, the restriction of the assignment $f\mapsto (\Delta f,f_*)$
	to $\mathrm{Diff}^+(X,\partial X)$ descends to a well-defined group homomorphism
	$[f]\mapsto \Delta f$
	of $\mathrm{Mod}(X,\partial X)$ to 
	$H_m(X;\Integral)\otimes H_m(X;\Integral)$ (Corollary \ref{variation_properties_corollary}).
	
	It remains to show that $\Delta f\in H_m(X;\Integral)\otimes H_m(X;\Integral)$
	is either alternating or symmetric, depending on the parity of $m$,
	for any $[f]\in\mathrm{Mod}(X,\partial X)$.
	
	To this end, we consider the oriented, connected, closed $2m$--manifold
	$$W_X=X\cup_{\partial X}(-X),$$
	obtained by doubling of $X$ along $\partial X$. 
	It suffices to work with rational coefficients below,
	treating $c_f$ as living in 
	$H_m(X;\Rational)\otimes_\Rational H_m(X;\Rational)$,
	as $H_m(X;\Integral)$ is torsion-free.
	Fix any basis 
	$$\alpha_1,\ldots,\alpha_r$$ 
	for $H_m(X;\Integral)$.
	We also regard
	$\alpha_1,\ldots,\alpha_r$
	as homology classes in $H_m(X;\Rational)$, 
	or in	$H_m(W_X;\Rational)$ via the inclusion or $\partial X\to W_X$.
	Denote by 
	$$\beta_1,\ldots,\beta_r$$ the dual basis of 
	$H_m(X,\partial X;\Rational)\cong H^m(X;\Rational)\cong \mathrm{Hom}(H_m(X;\Rational),\Rational))$,
	with respect to the orientation of $X$.
	By doubling rational relative $m$--cycle representatives of $\beta_1,\ldots,\beta_r$,
	we obtain homology classes 
	$$\hat{\beta}_1,\ldots,\hat{\beta}_r$$ in $H_m(W_X;\Rational)$,
	which map to $\beta_1,\ldots,\beta_r$ under the natural homomorphism
	of $H_m(W_X;\Rational)$ to $H_m(W_X,-X;\Rational)\cong H_m(X,\partial X;\Rational)$,
	respectively.	
	The pairwise algebraic intersection numbers in $W_X$ between
	the	homology classes $\alpha_1,\ldots,\alpha_r,\hat{\beta}_1,\ldots,\hat{\beta}_r$
	in $H_m(W_X;\Rational)$	are evidently 
	$$I(\alpha_i,\alpha_j)=I\left(\hat{\beta}_i,\hat{\beta}_j\right)=0,$$
	and 
	$$I\left(\alpha_i,\hat{\beta}_j\right)=(-1)^m\cdot I\left(\hat{\beta}_j,\alpha_i\right)=
	\begin{cases} 1 & i=j \\ 0 & i\neq j.\end{cases}$$
	Note that $H_m(W_X;\Rational)$ has dimension at most $2r$,
	which is the dimension sum of
	$H_m(X;\Rational)$ and $H_m(X,\partial X;\Rational)\cong H_m(W_X,-X;\Rational)$,
	by the homology long exact sequence for the pair $(W_X,X)$.
	It follows that
	$\alpha_1,\ldots,\alpha_r,\hat{\beta}_1,\ldots,\hat{\beta}_r$
	form a basis of $H_m(W_X;\Rational)$ over $\Rational$.
	
	For any $[f]\in\mathrm{Mod}(X,\partial X)$,
	we extend any representative $f\in\mathrm{Diff}^+(X,\partial X)$
	by the identity on $-X$, still denoted as $f\in\mathrm{Diff}^+(W_X)$.
	By computing algebraic intersection numbers with the basis vectors,
	the induced linear automorphism $f_*\in\mathrm{GL}(H_m(W_X;\Rational))$
	takes the form
	$$f_*(\alpha_i)=\alpha_i,
	\mbox{ and }
	f_*(\hat{\beta}_i)=\hat{\beta}_i+\sum_{j=1}^r c_{ij}\alpha_j,$$
	for some coefficients $c_{ij}\in\Rational$.
	According to (\ref{Delta_f_def}),  $(\Delta f)(\beta_i)\in H_m(X;\Rational)$
	agrees with the image of $f_*(\hat{\beta}_i)$ under the homomorphism 
	$H_m(W_X;\Rational)\to H_m(X;\Rational)$ induced by the obvious folding map
	$W_X\to X$,
	implying
	$$(\Delta f)(\beta_i)=\sum_{j=1}^r c_{ij}\alpha_j.$$
	This is equivalent to the formula
	$$\Delta f=\sum_{i=1}^r\sum_{j=1}^r c_{ij}\,\alpha_i\otimes \alpha_j,$$
	identifying $\mathrm{Hom}(H_m(X,\partial X;\Integral),H_m(X;\Integral))$
	with $H_m(X;\Integral)\otimes H_m(X;\Integral)$.	
	
	For all $i,j\in\{1,\ldots,r\}$,	the identity
	$$I\left( f_*(\hat{\beta}_i),f_*(\hat{\beta}_j)\right)=I\left(\hat{\beta}_i,\hat{\beta}_j\right)$$
	implies
	$$(-1)^m\cdot c_{ij}+ c_{ji}=0.$$
	It follows that the variation tensor $\Delta f$ in $H_m(X;\Integral)\otimes H_m(X;\Integral)$
	is alternating if $m$ is even, or symmetric if $m$ is odd, as asserted. 
\end{proof}

\begin{corollary}[the homological action formula for an implantation]\label{implanted_homological}
	Assume $X$ as in Proposition \ref{variation_tensor}.
	Suppose $[f]\in\mathrm{Mod}(X,\partial X)$ and 
	$\Delta f=\alpha_1\otimes\beta_1+\ldots+\alpha_s\otimes\beta_s$
	in $H_m(X;\Integral)\otimes H_m(X;\Integral)$.
	Then, for any oriented, connected, compact smooth $2m$--manifold $Y$,
	and for any orientation-preserving smooth embedding $i\colon X\to Y$,
	the following formula holds for all $\eta\in H_m(Y,\partial Y;\Integral)$.
	$$(i_*(f))_*(\eta)=\eta+I\left(i_*(\alpha_1),\eta\right)\,i_*(\beta_1)+\cdots+	I\left(i_*(\alpha_s),\eta\right)\,i_*(\beta_s).$$
\end{corollary}

\begin{proof}
For simplicity, we identify $X$ with its image $i(X)\subset Y$, 
and $i\colon X\to Y$ as the inclusion.
Denote
$$\eta|_X\in H_m(X,\partial X;\Integral)$$
the restriction of $\eta$ to $X$.
Namely, $\eta|_X$ is the image of $\eta$ under the natural homomorphism
$H_m(Y,\partial Y;\Integral)\to H_m(Y,Y\setminus\mathrm{int}(X);\Integral)$
followed by the excision isomorphism 
$H_m(Y,Y\setminus \mathrm{int}(X);\Integral)\cong H_m(X,\partial X;\Integral)$.
The asserted formula is obtained by the following computation in $H_m(Y,\partial Y;\Integral)$:
\begin{eqnarray*}
\left(i_*(f)\right)_*(\eta)&=&\eta+\left(\Delta (i_*(f))\right)(\eta)\\
&=&\eta+i_*\left((\Delta f)(\eta|_X)\right)\\
& =& \eta + I_X\left(\alpha_1,\eta|_X\right)\,i_*(\beta_1)+\cdots+I_X\left(\alpha_s,\eta|_X\right)\,i_*(\beta_s)\\
&=& \eta + I_Y\left(i_*(\alpha_1),\eta\right)\,i_*(\beta_1)+\cdots+I_Y\left(i_*(\alpha_s),\eta\right)\,i_*(\beta_s).
\end{eqnarray*}
Here, the first step follows from the constructive form (\ref{Delta_f_def}) of the variation;
the second step follows from Proposition \ref{variation_properties};
the third step follows from Proposition \ref{variation_tensor};
the last step follows from the naturality of the algebraic intersection number pairing;
the subscripts of $I$ indicate the manifolds in reference.
\end{proof}

\begin{remark}\label{implanted_homological_remark}
	For any $\eta\in H_m(Y;\Integral)$, the formula for $(i_*(f))_*(\eta)\in H_m(Y;\Integral)$
	takes the exactly similar form as the formula in Corollary \ref{implanted_homological}.
	In fact, it can be derived easily from Corollary \ref{implanted_homological},
	applying to $W_Y=Y\cup_{\partial Y}(-Y)$.
	Note that $H_*(W_Y;\Integral)\to H_*(W_Y,Y;\Integral)$ is surjective (by doubling relative cycles),
	and hence, $H_*(Y;\Integral)\to H_*(W_Y;\Integral)$ is injective.
\end{remark}

\begin{example}[formulas for an annulus Dehn twist]\label{Dehn_example}
	Let $\mathcal{A}=c\times[-1,1]$ be an oriented annulus 
	with a simple closed core curve $c=c\times0$. 
	Denote by $T_c\in\mathrm{Mod}(\mathcal{A},\partial \mathcal{A})$
	be the (right-hand) Dehn twist along $c$.
	It is easy to see
	$$(\Delta T_c)(\xi)=I([c],\xi)\,[c],$$
	for any $\xi$ in $H_1(\mathcal{A},\partial\mathcal{A};\Integral)\cong\Integral$,
	(regardless the orientation of $c$).
	Here, we abuse the notation $\Delta T_c$,
	as the variation does not depend on the choice
	of a representative of $T_c$ in $\mathrm{Diff}^+(\mathcal{A},\partial\mathcal{A})$.
	Equivalently,
	$$\Delta T_c=[c]\otimes[c]=[c]^2$$
	as a symmetric tensor in $\odot^2 H_1(\mathcal{A};\Integral)\cong\Integral$ (Proposition \ref{variation_tensor}).
	For any oriented, connected, compact, smooth surface $Y$,
	and any orientation-preserving smooth embedding $i\colon\mathcal{A}\to Y$,
	we recover the familiar formula
	$$\left(i_*(T_c)\right)_*(\eta)=\eta+I\left(i_*([c]),\eta\right)\,i_*([c]),$$
	for any $\eta\in H_1(Y,\partial Y;\Integral)$ (Corollary \ref{implanted_homological}).
\end{example}

\begin{example}[formulas for a model barbell twist]\label{BG_Delta_f_example}
	Let $(\mathcal{N},\Gamma,\varphi)=(\mathcal{N},P_a,P_b,I,\varphi)$
	be a model tuple (Definition \ref{BG_barbell_twist_def}).
	For simplicity, we fix a representative 
	$g\in\mathrm{Diff}^+(\mathcal{N},\partial\mathcal{N})$
	of the model barbell twist $\varphi\in\mathrm{Mod}(\mathcal{N},\partial\mathcal{N})$.
	as described in Theorem \ref{BG_construction}.
	Let 
	$(E_a,\partial E_a),(E_b,\partial E_b)\subset(\mathcal{N},\partial\mathcal{N})$ 
	be neatly embedded $2$--disks as in Theorem \ref{BG_construction}.
	The defining formula (\ref{Delta_f_def}) implies directly
	$(\Delta g)([E_a])=[P_b]$ and  $(\Delta g)(E_b)=-[P_a]$,
	for the basis $[E_a],[E_b]$ of $H_2(\mathcal{N},\partial \mathcal{N};\Integral)\cong\Integral^2$. 
	Hence, for any $\xi\in H_2(\mathcal{N},\partial \mathcal{N};\Integral)$,
	\begin{equation}\label{Delta_BG}
	(\Delta \varphi)(\xi)=I([P_a],\xi)\,[P_b]-I([P_b],\xi)\,[P_a],
	\end{equation}
	abusing the notation $\Delta \varphi=\Delta g$.
	Equivalently,
	\begin{equation}\label{Delta_BG_tensor}
	\Delta \varphi=[P_a]\otimes[P_b]-[P_b]\otimes[P_a]=[P_a]\wedge[P_b],
	\end{equation}	
	as an alternating tensor in $\wedge^2 H_2(\mathcal{N};\Integral)\cong\Integral$ (Proposition \ref{variation_tensor}).
	
	Let $Y$ be an oriented, connected, compact $4$--manifold,
	and $i\colon\mathcal{N}\to Y$
	be a smooth embedding.
	We obtain
	\begin{equation}\label{Delta_BG_implanted_homological}
	\left(i_*(\varphi)\right)_*(\eta)=\eta+I\left(i_*([P_a]),\eta\right)\,i_*\left([P_b])-I(i_*([P_b]),\eta\right)\,i_*([P_a]),
	\end{equation}
	for any $\eta\in H_2(Y,\partial Y;\Integral)$ (Corollary \ref{implanted_homological}).
	In particular, we recover the homological formula in Corollary \ref{BG_construction_corollary}
	as a special case of (\ref{Delta_BG_implanted_homological}) for $Y=W_\mathcal{N}$.
\end{example}

\section{Extendability for boundary diffeomorphisms of a 1-handlebody}\label{Sec-1hdlbdy}
In this section, we prove an extension theorem 
for orientation-preserving boundary differomorphism
of a $4$--dimensional $1$--handlebody.

\begin{theorem}\label{extendability_1hdlbdy}
Let $Z$ be an oriented smooth $4$--manifold
diffeomorphic to $(S^1\times D^3)^{\natural n}$,
for some integer $n\geq0$.
Then the natural homomorphism
$$\mathrm{Diff}^+(Z)\to \mathrm{Diff}^+(\partial Z)\colon f\mapsto f|_{\partial Z}$$
is surjective.
\end{theorem}

\begin{corollary}\label{extendability_1hdlbdy_corollary}
	If $Z_0$ and $Z_1$ are oriented smooth $4$--manifolds,
	both diffeomorphic to $(S^1\times D^3)^{\natural n}$ for some integer $n\geq0$,
	then any orientation-preserving diffeomorphism $\partial Z_1\to \partial Z_0$
	admits an extension as an orientation-preserving diffeomorphism $Z_1\to Z_0$.
\end{corollary}

\begin{proof}
	Fix some reference orientation-preserving diffeomorphism $Z_1\to Z_0$.
	Any orientation-preserving diffeomorphism $\partial Z_1\to \partial Z_0$
	differs from the boundary restriction of the reference diffeomorphism
	by post-composition with a boundary diffeomorphism 
	in $\mathrm{Diff}^+(\partial Z_0)$.
	By Theorem \ref{extendability_1hdlbdy},
	any boundary diffeomorphism in $\mathrm{Diff}^+(\partial Z_0)$
	arises as the boundary restriction of some diffeomorphism in $\mathrm{Diff}^+(Z_0)$.
	It follows that any orientation-preserving diffeomorphism $\partial Z_1\to \partial Z_0$
	extends to be an orientation-preserving diffeomorphism $Z_1\to Z_0$, as asserted.	
\end{proof}

The rest of this section is devoted to the proof 
of Theorem \ref{extendability_1hdlbdy}.

\begin{lemma}\label{extend_cut_system}
Let $Z$ be a smooth $4$--manifold
diffeomorphic to $(S^1\times D^3)^{\natural n}$,
for some integer $n\geq0$.
Suppose that 
$$\mathcal{S}\subset\partial Z$$
is a union of $n$ smoothly embedded, mutually disjoint $2$--spheres,
such that $(\partial Z)\setminus\mathcal{S}$ is connected.
Then,
there exists a union of $n$ neatly embedded, mutually disjoint $3$--disks
$$\mathcal{B}\subset Z$$
with $\partial \mathcal{B}= \mathcal{S}$,
such that for any tubular neighborhood $\mathcal{B}\times[-1,1]$ 
of $\mathcal{B}$ in $Z$,
the region $Z\setminus (\mathcal{B}\times(-1,1))$
is diffeomorphic to $D^4$,
after corner smoothing along the locus $\partial\mathcal{B}\times\{-1,1\}$.
\end{lemma}

\begin{proof}
We start by recalling two well-known correspondences in $3$--manifold topology.

For any $3$--manifold $U$ diffeomorphic to $(S^1\times D^2)^{\natural n}$,
the homotopy set $[U,\vee^n S^1]$ corresponds naturally and bijectively
to the diffeotopy class of $n$ neatly embedded, mutually disjoint,
ordered, transversely oriented, and unionwise nonseparating $2$--disks
$$\Delta_1\sqcup\cdots\sqcup \Delta_n\subset U.$$
More precisely, 
taking any tubular neighborhood $(\Delta_1\sqcup\cdots \Delta_n)\times[-1,1]$
in $U$ respecting the transverse orientation,
the corresponding homotopy class is represented by a continuous map 
$U\to \vee^n S^1$,
such that $U\setminus ((\Delta_1\sqcup\cdots\sqcup \Delta_n)\times[-1,1])$ 
maps to the wedge point,
and each $\Delta_i\times[-1,1]$ maps onto the $i$--th oriented circle $S^1$,
factoring through the projection onto $[-1,1]$ and respecting the transverse orientation.

For any $3$--manifold $M$ is diffeomorphic to $(S^1\times S^2)^{\# n}$.
The homotopy set $[M,\vee^n S^1]$ corresponds naturally and bijectively
to the the diffeotopy class of $n$ smoothly embedded, mutually disjoint,
ordered, transversely oriented, and unionwise nonseparating $2$--spheres
$$\Sigma_1\sqcup\cdots\sqcup \Sigma_n\subset M.$$
More precisely, 
taking any tubular neighborhood 
$(\Sigma_1\sqcup\cdots \Sigma_n)\times[-1,1]$
in $M$ respecting the transverse orientation,
the corresponding homotopy class is represented by a continuous map 
$M\to \vee^n S^1$,
such that $M\setminus ((\Sigma_1\sqcup\cdots\sqcup \Sigma_n)\times[-1,1])$ 
maps to the wedge point,
and each $\Sigma_i\times[-1,1]$ maps onto the $i$--th oriented circle $S^1$,
factoring through the projection onto $[-1,1]$ and respecting the transverse orientation.

Applying the above facts, we prove Lemma \ref{extend_cut_system}, as follows.

Let $Z$ be a smooth $4$--manifold
diffeomorphic to $(S^1\times D^3)^{\natural n}$,
for some integer $n\geq0$.
We view $Z$ diffeomorphically as $U\times[-1,1]$, after corner smoothing along $\partial U\times\{-1,1\}$,
where $U$ is diffeomorphic to $(S^1\times D^2)^{\natural n}$.
For any $\mathcal{S}\subset\partial Z$ as assumed in Lemma \ref{extend_cut_system},
we fix some ordering and transverse orientations of the component $2$--spheres of $\mathcal{S}$.
Then, the diffeotopy class of $\mathcal{S}$ in $\partial Z$ (diffeomorphic to $(S^1\times S^2)^{\# n}$)
corresponds to the homotopy class of a continuous map $\partial Z\to \vee^n S^1$.
The restriction of the continuous map to $U\times 1$
determines a continuous map $U\to \vee^n S^1$.
The homotopy class of the continuous map $U\to \vee^n S^1$
corresponds to some diffeotopy class
of $n$ neatly embedded, mutually disjoint $2$--disks
$\Delta_1\sqcup\cdots\sqcup \Delta_n\subset U$,
which are ordered, transversely oriented, and unionwise nonseparating.
In $Z=U\times[-1,1]$, we construct
$$\mathcal{B}=(\Delta_1\sqcup\cdots\sqcup\Delta_n)\times[-1,1].$$

By construction, $\mathcal{B}$ forms
a union of $n$ neatly embedded, mutually disjoint $3$--disks in $Z$,
such that for some tubular neighborhood $\mathcal{B}\times[-1,1]$ 
of $\mathcal{B}$ in $Z$,
the region $Z\setminus (\mathcal{B}\times(-1,1))$
is diffeomorphic to $D^4$,
after corner smoothing (along the locus $\partial\mathcal{B}\times\{-1,1\}$).
This is because $U\setminus ((\Delta_1\sqcup\cdots\sqcup\Delta_n)\times(-1,1))$
is diffeomorphic to $D^3$, after corner smoothing 
(along the locus $(\partial \Delta_1\sqcup\cdots\sqcup\partial \Delta_n)\times\{-1,1\}$).
It follows that the same property holds for any tubular neighborhood of $\mathcal{B}$ in $Z$,
by uniqueness of tubular neighborhoods
(see \cite[Theorem 2.5.5 and Proposition 2.5.8]{Wall_book}).
Moreover, 
the diffeotopy classes of $\partial\mathcal{B}$ in $\partial Z$
corresponds to the same homotopy class of continous maps
$\partial Z\to \vee^n S^1$ as that of $\mathcal{S}$ in $\partial Z$.
This is because the inclusion $U\times 1\to \partial Z$ is $\pi_1$--isomorphic,
and $\vee^n S^1$ is aspherical, 
according to our description of the correspondence above.
It follows that $\partial\mathcal{B}$ and $\mathcal{S}$
are diffeotopic in $\partial Z$.
Therefore, possibly after modifying $\mathcal{B}$
by a diffeotopy of $Z$ supported near $\partial Z$,
we can further make sure $\partial\mathcal{B}=\mathcal{S}$,
without affecting the other achieved properties.
This yields some $\mathcal{B}\subset Z$ 
for any given $(Z,\mathcal{S})$, as asserted.
\end{proof}

With the above preparation, we prove Theorem \ref{extendability_1hdlbdy} as follows.

Let $Z$ be an oriented smooth $4$--manifold
diffeomorphic to $(S^1\times D^3)^{\natural n}$,
for some integer $n\geq0$.
Fix some union of $n$ neatly embedded, mutually disjoint $3$--disks
$$\mathcal{B}\subset Z,$$
together with some tubular neighborhood $\mathcal{B}\times[-1,1]$ in $Z$,
such that the region
$Z\setminus (\mathcal{B}\times(-1,1))$
is diffeomorphic to $D^4$, 
after corner smoothing along the locus $\partial \mathcal{B}\times\{-1,1\}$.

For any $g\in\mathrm{Diff}^+(\partial Z)$,
we apply Lemma \ref{extend_cut_system},
obtaining another union of $n$ neatly embedded, mutually disjoint $3$--disks
$$\mathcal{B}'\subset Z$$
with 
$$\partial \mathcal{B}'=g(\partial \mathcal{B}),$$
and 
some tubular neighborhood $\mathcal{B}'\times[-1,1]$ in $Z$,
such that 
$Z\setminus(\mathcal{B}'\times(-1,1))$
is diffeomorphic to $D^4$ after corner smoothing.
We can require
$(\partial \mathcal{B}')\times[-1,1]=g((\partial \mathcal{B})\times[-1,1])$,
respecting the parametrization fiberwise,
by uniqueness of tubular neighborhoods in differential topology
(see \cite[Theorem 2.5.5 and Proposition 2.5.8]{Wall_book}).
Therefore, we can extend $g$ over the tubular region
$\mathcal{B}\times[-1,1]$,
such that $\mathcal{B}\times[-1,1]$ maps diffeomorphically onto
$\mathcal{B}'\times[-1,1]$.
Since $\mathrm{Mod}(S^3)$ is trivial \cite{Hatcher_smale},
the obtained diffeomorphism of the $3$--sphere
$\partial(Z\setminus(\mathcal{B}\times (-1,1)))$
onto the $3$--sphere
$\partial(Z\setminus(\mathcal{B}'\times (-1,1)))$
can be further extended
to be a diffeomorphism of the $4$--disk
$Z\setminus(\mathcal{B}\times (-1,1))$
onto the $4$--disk
$Z\setminus(\mathcal{B}'\times (-1,1)))$.

Therefore, we obtain an extension of $g\in\mathrm{Diff}^+(\partial Z)$,
as a diffeomorphism 
$$f\in\mathrm{Diff}^+(Z),$$ 
namely,
$$g=f|_{\partial Z},$$
as desired.

This completes the proof of Theorem \ref{extendability_1hdlbdy}.

\section{Light bulb tricks and 2-handle embeddings}\label{Sec-2hdls}
In this section, we prove a generalization of 
the ($4$--dimensional) light bulb theorem for several disks (Theorem \ref{lbt_disks}).
We derive an embedding uniqueness theorem for several $2$--handles
in a simplified situation (Corollary \ref{uniqueness_bnb_2hdls}),
which suffices for our proof of Theorem \ref{main_T} in Section \ref{Sec-main_proof}.

\begin{definition}\label{bnb_terms}
Let $X$ be a connected, orientable, smooth $4$--manifold with possibly disconnected boundary.
\begin{enumerate}
\item
For any integer $n\geq1$,
a union of $n$ neatly embedded, mutually disjoint $2$--disks
$$(\mathcal{D},\partial\mathcal{D})\subset (X,\partial X)$$
is said to be \emph{boundary $\pi_1$--nonbusting},
if the inclusion $\partial X\setminus \partial\mathcal{D}\to \partial X$
is $\pi_1$--isomorphic restricted to each connected component.
\item
For any smooth embedding of several $4$--dimensional $2$--handles
$$i\colon \left(\sqcup^n(D^2\times D^2),\sqcup^n(S^1\times D^2)\right)\to (X,\partial X),$$
we say that $i$ is \emph{boundary $\pi_1$--nonbusting},
if the image of the core $2$--disks $i(\sqcup^n(D^2\times o))$ 
form a boundary $\pi_1$--nonbusting union of $2$--disks in $X$.
\end{enumerate}
\end{definition}

\begin{proposition}\label{bnb_geometric}
	Let $(\mathcal{D},\partial\mathcal{D})\subset (X,\partial X)$
	be a boundary $\pi_1$--nonbusting union 
	of $n$ neatly embedded, mutually disjoint $2$--disks,
	as in Definition \ref{bnb_terms}.
	
	Then, $X$ admits some smooth boundary-connect sum decomposition
	with summands $X_0$ and $n$ mutually disjoint copies of $D^2\times S^2$,
	such that each component $2$--disk of $\mathcal{D}$
	sits in a distinct summand $D^2\times S^2$ as a fiber $D^2$ therein,
	and stays away from the decomposition $3$--disk of that summand.
\end{proposition}

\begin{proof}
	Enumerate the component $2$--disks of $\mathcal{D}$ as 
	$$D_1,\ldots,D_n\subset X$$
	By assumption, the inclusion 
	$\partial X\setminus \partial\mathcal{D}\to \partial X$
	is $\pi_1$--isomorphic restricted to each connected component.
	By the disk theorem in $3$--manifold topology,
	we obtain $n$ smoothly embedded, mutually disjoint $2$--spheres
	$$G_1,\ldots,G_n\subset\partial X,$$
	such that each $G_j$ intersects $\partial\mathcal{D}$
	transversely in $\partial X$ at a unique point in $\partial D_j$.
	
	Some neighborhood of $D_j\cup G_j$ in $X$
	can be viewed identically as a neighborhood of $(D_j\times q_j)\cup(p_j\times G_j)$
	in $D_j\times G_j$, for some $(p_j,q_j)\in D_j\times G_j$. 
	With respect to a product Riemannian metric on $D_j\times G_j$,
	some sufficiently small distance--$\epsilon$
	neighborhood $(D_j\times q_j)\cup(p_j\times G_j)$
	forms a region therein.
	The intersection of the region with $\partial D_j\times G_j$ 
	is a $3$--submanifold diffeomorphic to $S^1\times S^2$ with a ($3$--disk) hole.
	The frontier of the region in $D_j\times G_j$ is
	diffeomorphic to $D^3$. 
	The region itself is diffeomorphic to $D^2\times S^2$
	after corner smoothing along the boundary of the frontier.
	
	Therefore, coming back to $X$, we see that each $D_j\cup G_j$
	is contained in a boundary-connect summand diffeomorphic to $D^2\times S^2$,
	such that $D_j$ sits as a fiber $D^2$ therein, 
	and stays away from the frontier $3$--disk.
	The asserted summand $X_0$ is obtained from $X$
	by removing the above $n$ summands $D^2\times S^2$.
\end{proof}

\begin{theorem}[{the $4$--dimensional light bulb theorem for several disks}]\label{lbt_disks}
Let $X$ be a connected, compact, orientable smooth $4$--manifold 
with possibly disconnected boundary.
Suppose that 
$$(\mathcal{D}_0,\partial\mathcal{D}_0),(\mathcal{D}_1,\partial\mathcal{D}_1)\subset (X,\partial X)$$
are boundary $\pi_1$--nonbusting unions of neatly embedded, mutually disjoint $2$--disks,
such that $\mathcal{D}_0$ and $\mathcal{D}_1$
coincide on some neighborhood of $\partial X$ in $X$.

Then, 
$\mathcal{D}_0$ and $\mathcal{D}_1$ are smoothly isotopic in $X$ 
relative to $\partial \mathcal{D}_0=\partial\mathcal{D}_1$,
if and only if 
every component $2$--disk $D_0\subset\mathcal{D}_0$ is smoothly isotopic in $X$ 
to some component $2$--disk $D_1\subset \mathcal{D}_1$ relative to $\partial D_0=\partial D_1$.

Moreover, this occurs if and only if
for each pair of component $2$--disks $(D_0,D_1)$ from $(\mathcal{D}_0,\mathcal{D}_1)$
coincident near $\partial X$, 
$D_0$ is homotopic to $D_1$ in $X$ relative to $\partial D_0=\partial D_1$,
and meanwhile, the relative Dax invariant $\mathrm{Dax}(D_0,D_1)$ with respect to $X$ vanishes.
\end{theorem}

\begin{proof}
	The ``moreover'' part follows directly
	from the single disk version of the light bulb theorem, 
	due to Kosanovi\'c and Teichner \cite[Theorem 1.1]{KT_lbt} 
	(see Remark \ref{lbt_disks_remark} below).
	The ``only if'' direction of the main assertion is trivial. 
	Therefore, it suffices to prove the ``if'' direction of the main assertion.
	
	Suppose that 
	$$(\mathcal{D}_0,\partial\mathcal{D}_0),(\mathcal{D}_1,\partial\mathcal{D}_1)\subset (X,\partial X)$$
	are boundary $\pi_1$--nonbusting unions of neatly embedded, mutually disjoint $2$--disks,
	such that $\mathcal{D}_0$ and $\mathcal{D}_1$
	coincide on some neighborhood of $\partial X$ in $X$.
	We assume that 
	every component $2$--disk $D_0\subset\mathcal{D}_0$ is smoothly isotopic in $X$ 
	to some component $2$--disk $D_1\subset \mathcal{D}_1$ relative to $\partial D_0=\partial D_1$.

	Pick some pair of $2$--disks $D_0$ and $D_1$ as above.
	By diffeotopy extension, there exists some $h\in\mathrm{Diff}_0(X,\partial X)$
	with the property $h(D_1)=D_0$. 
	Possibly after replacing $\mathcal{D}_1$ with $h(\mathcal{D}_1)$ as above,
	we may assume 
	$$D_0=D_1.$$
	
	Set 
	$\mathcal{D}'_0=\mathcal{D}_0\setminus D_0,\mbox{ and }\mathcal{D}'_1=\mathcal{D}_1\setminus D_1$.	
	By Proposition \ref{bnb_geometric}, we can rewrite $X$ as a boundary-connect sum
	$$X=X'\,\natural\,V,$$
	where $V$ is diffeomorphic to $D^2\times S^2$.
	Moreover, we can assume $D_0=D_1$ to be neatly embedded in $V$,
	and $\mathcal{D}'_0$ and $\mathcal{D}'_1$ to be neatly embedded in $X'$,
	(both away from the decomposition $3$--disk).
	Because of Proposition \ref{bnb_geometric},
	we observe that
	$$(\mathcal{D}'_0,\partial\mathcal{D}'_0),(\mathcal{D}'_1,\partial\mathcal{D}'_1)\subset(X',\partial X')$$
	still form boundary $\pi_1$--nonbusting unions with respect to $X'$.
	
	By assumption, for any component $2$--disk $D'_0\subset\mathcal{D}'_0$,
	there exists some $D'_1\subset\mathcal{D}'_1$ with $\partial D'_0=\partial D'_1$,
	such that $D'_0$ and $D'_1$ are smoothly isotopic in $X$ relative to $\partial D'_0$.
	
	We claim that $D'_0$ and $D'_1$ are also smoothly isotopic in $X'$ relative to $\partial D'_0$.
	In fact, by attaching a $4$--dimensional $3$--handle $U$ along a factor $2$--sphere 
	on the boundary of the boundary-connect summand $V$,
	we obtain an inclusion
	$$X\to X\cup U.$$ 
	It follows that there exists some smooth isotopy
	of $D'_0$ to $D'_1$ in $X\natural U$ relative to $\partial D'_0$.
	Since $V\cup U$ is obviously diffeomorphic to $D^4$, the enlarged smooth $4$--manifold 
	$$X\cup U=X'\,\natural\,(V\cup U)$$
	is diffeomorphic to $X'$,
	and $V\cup U$ is contained in some collar neighborhood of $\partial (X\cup U)$.
	It follows that there exists some smooth isotopy of $D'_0$ to $D'_1$ relative to $\partial D'_0$,
	in $X'\natural (V\cup U)$ and away from a collar neighborhood containing $V\cup U$.
	This yields a smooth isotopy of $D'_0$ to $D'_1$ relative to $\partial D'_0$ in $X'$, as claimed.
	
	The above argument allows us to reduce the case with $(X,\mathcal{D}_0,\mathcal{D}_1)$
	to the case with $(X',\mathcal{D}'_0,\mathcal{D}'_1)$,
	which satisfies the same assumptions, and involves fewer $2$--disks.
	Therefore, the proof completes by induction.	
\end{proof}

\begin{remark}[the relative Dax invariant]\label{lbt_disks_remark}
	We recall basic facts about the relative Dax invariant for the reader's reference.
	Let $X$ be a connected, oriented, smooth $4$--manifold with boundary.
	For neatly embedded $2$--disks 
	$$(D_0,\partial D_0),(D_1,\partial D_1)\subset(X,\partial X)$$
	coincident near $\partial X$,
	and homotopic in $X$ relative to $\partial D_0=\partial D_1$.
	Fixing a basepoint on $\partial D_0$, the \emph{relative Dax invariant} for $(D_0,D_1)$
	is a well-defined element
	$$\mathrm{Dax}(D_0,D_1)\in\Integral[\pi_1(X)\setminus 1]/d_3(\pi_3(X)).$$
	Here, $\Integral[\pi_1(X)\setminus 1]$ denotes the additive abelian group
	consisting of all elements of the form $\sum_g a_g\cdot g$,
	summing over nontrivial elements $g\in\pi_1(M)$,
	and with only finitely many nonzero coefficients $a_g\in\Integral$;
	$d_3(\pi_3(X))$ denotes the image of an abelian group homomorphism
	$$d_3\colon \pi_3(X)\to \Integral[\pi_1(X)\setminus 1],$$
	called the \emph{Dax homomorphism} for $\pi_3(X)$,
	which depends only on $X$ (with the base point).
	The relative Dax invariant is additive, namely,
	$$\mathrm{Dax}(D_0,D_1)+\mathrm{Dax}(D_1,D_2)=\mathrm{Dax}(D_0,D_2),$$
	and anticommutative, namely,
	$$\mathrm{Dax}(D_0,D_1)=-\mathrm{Dax}(D_1,D_0).$$
	If $D_0$ and $D_1$ are smoothly isotopic relative to the boundary,
	$$\mathrm{Dax}(D_0,D_1)=0.$$
	We refer the reader to \cite[Section 3]{Gabai_self-ref} 
	for an elementary treatment due to Gabai, 
	and for general background on early developments.	
	
	Under the extra assumption that $D_0,D_1$ are boundary $\pi_1$--nonbusting in $X$,
	Konosivi\'c and Teichner prove that 
	$D_0$ and $D_1$ are smoothly isotopic relative to the boundary,
	if and only if the relative Dax invariant $\mathrm{Dax}(D_0,D_1)$ vanishes
	\cite[Theorem 1.1]{KT_lbt}.
	This is their (single) \emph{disk version} of the $4$--dimensional light bulb theorem.
	See also Cha and Kim \cite{CK_lbt} for a generalization to 
	properly topologically embedded disks. 
	Formerly, the \emph{$4$--dimensional light bulb theorem}	
	is established by Gabai \cite{Gabai_lbt},
	generalizing the classical light bulb theorem in $3$--manifold topology.
	Gabai's theorem is about smoothly embedded $2$--spheres
	(with so-called smooth geometric duals).
	In the same paper, 
	Gabai also proves the \emph{several spheres version} \cite[Theorem 10.1]{Gabai_lbt},
	based on the \emph{single sphere version} \cite[Theorem 1.9]{Gabai_lbt}.	
\end{remark}

\begin{corollary}\label{uniqueness_bnb_2hdls}
Let $X$ be a simply connected, compact, smooth $4$--manifold 
with possibly disconnected boundary.
Suppose that 
$$i_0,i_1\colon \left(\sqcup^n(D^2\times D^2),\sqcup^n(S^1\times D^2)\right)\to (X,\partial X)$$
are boundary $\pi_1$--nonbusting, smooth embeddings of several $4$--dimensional $2$--handles,
for some integer $n\geq1$,
such that $i_0$ and $i_1$ coincide 
on some neighborhood of $\sqcup^n(S^1\times D^2)$ in $\sqcup^n(D^2\times D^2)$.

If $i_0$ and $i_1$ are homotopic relative to $\sqcup^n(S^1\times D^2)$,
then there exists some $h\in\mathrm{Diff}_0(X,\partial X)$,
such that $h\circ i_1=i_0$ holds on $\sqcup^n(D^2\times D^2)$. 
\end{corollary}

\begin{proof}
Note that simply connected manifolds are all orientable.
Since $\pi_1(X)$ is trivial, the relative Dax invariant vanishes automatically
for any pair of relatively homotopic, neatly embedded $2$--disks (Remark \ref{lbt_disks_remark}).
Under the $\pi_1$--nonbusting assumption, 
the (image of the) disjoint unions of the $n$ core $2$--disks
as embedded via $i_0$ and $i_1$ can be matched up by some smooth isotopy
relative to the boundary circles, 
by the $4$--dimensional lightbulb theorem for several disks (Theorem \ref{lbt_disks}).
This implies the existence of some $h\in\mathrm{Diff}_0(X,\partial X)$
with $h\circ i_1=i_0$,
by standard techniques in differential topology.
\end{proof}

\section{Proof of the main theorem}\label{Sec-main_proof}

This section is devoted to the proof of Theorem \ref{main_T}.

\begin{lemma}\label{X_in_D}
	Let $X$ be a connected, compact, orientable $4$--manifold. 
	Suppose that there exists some smoothing embedding $i\colon X\to D^4$,
	Then, the following direct-product decomposition holds.
	$$\mathrm{Mod}(X,\partial X)=\mathrm{Mod}(D^4,\partial D^4)\times \mathrm{Ker}(i_*).$$
	Here, $\mathrm{Mod}(D^4,\partial D^4)$
	denotes the image of $j_*\colon \mathrm{Mod}(D^4,\partial D^4)\to \mathrm{Mod}(X,\partial X)$
	for any smooth embedding $j\colon D^4\to X$
	and with $i\circ j\colon D^4\to D^4$ orientation-preserving.
	In fact, $j_*$ is injective, and is independent of $j$ subject to the above requirements.
\end{lemma}

\begin{proof}
	Any smooth embedding $j\colon D^4\to X$ with $j(D^4)\subset\mathrm{int}(X)$
	forms a tubular neighborhood of the point $j(o)\in \mathrm{int}(X)$ in $X$.
	With the extra assumption that $i\circ j\colon D^4\to D^4$ is orientation-preserving,
	the tubular neighborhood uniqueness theorem \cite[Theorem 2.5.5]{Wall_book}
	implies directly 
	that $j_*\colon \mathrm{Mod}(D^4,\partial D^4)\to \mathrm{Mod}(X,\partial X)$
	is independent of $j$.
	(The independence of $j_*$ also holds 
	for any oriented $X$ and any orientation-preserving $j$,
	without assuming $X$ to be embeddable into $D^4$.) 
	Moreover, it also implies that the composite homomorphism $i_*\circ j_*\colon 
	\mathrm{Mod}(D^4,\partial D^4)\to\mathrm{Mod}(X,\partial X)\to\mathrm{Mod}(D^4,\partial D^4)$
	is the trivial automorphism of $\mathrm{Mod}(D^4,\partial D^4)$.
	This implies a semidirect-product decomposition 
	$\mathrm{Mod}(X,\partial X)=\mathrm{Ker}(i_*)\rtimes \mathrm{Mod}(D^4,\partial D^4)$.
	On the other hand, any $[f]\in\mathrm{Mod}(D^4,\partial D^4)$
	can be represented by some $f\in\mathrm{Diff}^+(X,\partial X)$
	fixing some neighborhood of $j(D^4)\subset\mathrm{int}(X)$.
	Therefore, $\mathrm{Mod}(D^4,\partial D^4)$ is contained in the center of $\mathrm{Mod}(X,\partial X)$.
	It follows that the above semidirect-product decomposition is a direct-product decomposition, as asserted.
\end{proof}

\begin{lemma}\label{Mod_nS2D2_rel}
	Let $X$ be an oriented, smooth $4$--manifold diffeomorphic to $(S^2\times D^2)^{\natural n}$,
	for some integer $n\geq0$.
	The following statements all hold.
	\begin{enumerate}
	\item 
	There exists some 
	orientation-preserving smoothing embedding $i\colon X\to D^4$ with $i(X)\subset\mathrm{int}(D^4)$,
	such that $D^4\setminus i(\mathrm{int}(X))$ is diffeomorphic to $D^4\# (S^1\times D^3)^{\natural n}$.
	\item
	For any orientation-preserving smooth embedding $i\colon X\to D^4$ as above,
	the assignment 
	$$[f]\mapsto (i_*([f]),\Delta f)$$
	determines a group isomorphism
	$$\mathrm{Mod}(X,\partial X)\to \mathrm{Mod}(D^4,\partial D^4)\times \wedge^2 H_2(X;\Integral).$$
	In particular, $\mathrm{Mod}(X,\partial X)$ is an abelian group.
	\item 
	The above group isomorphism depends only on $X$.
	In other words, 
	the implantation homomorphism
	$$i_*\colon \mathrm{Mod}(X,\partial X)\to \mathrm{Mod}(D^4,\partial D^4)$$
	does not depend on the choice of 
	an orientation-preserving smooth embedding $i\colon X\to D^4$ as above.
	\end{enumerate}
\end{lemma}

\begin{proof}
	The first statement in Lemma \ref{Mod_nS2D2_rel} is well-known.
	In general, $S^{p+q+1}$
	admits a standard (analogous) genus $n$ Heegaard splitting into 
	$(D^{p+1}\times S^q)^{\natural n}$ and $(S^p\times D^{q+1})^{\natural n}$,
	along a Heegaard hypersurface $S^p\times S^q$.
	Removing an open $(p+q+1)$--disk from the $(S^p\times D^{q+1})^{\natural n}$
	is equivalent to adding a connect summand $D^{p+q+1}$,
	yielding a splitting of $D^{p+q+1}$ into $(D^{p+1}\times S^q)^{\natural n}$
	and $D^{p+q+1}\#(S^p\times D^{q+1})^{\natural n}$.
	
	To prove the second statement in Lemma \ref{Mod_nS2D2_rel},
	we invoke the direct-product decomposition
	$$\mathrm{Mod}(X,\partial X)=\mathrm{Mod}(D^4,\partial D^4)\times \mathrm{Ker}(i_*),$$
	(Lemma \ref{X_in_D}).
	Note that the direct product factor $\mathrm{Mod}(D^4,\partial D^4)$
	maps identically onto $\mathrm{Mod}(D^4,\partial D^4)$, under
	the implantation homomorphism
	$$i_*\colon\mathrm{Mod}(X,\partial X)\to \mathrm{Mod}(D^4,\partial D^4),$$
	(Lemma \ref{X_in_D}).
	On the other hand, the variation group homomorphism
	$$\Delta\colon \mathrm{Mod}(X,\partial X)\to H_2(X;\Integral)\otimes H_2(X;\Integral)$$
	has kernel obviously containing the factor $\mathrm{Mod}(D^4,\partial D^4)$, 
	and has image contained in $\wedge^2 H_2(X;\Integral)$,
	by Proposition \ref{variation_tensor}.
	It remains to show that 
	the image of $\Delta$ is equal to $\wedge^2H_2(X;\Integral)$,
	and the kernel of $\Delta$ is equal to $\mathrm{Mod}(D^4,\partial D^4)$.
	
	To prove $\mathrm{Im}(\Delta)=\wedge^2H_2(X;\Integral)$,
	we can pick some collection of $n$ smoothly embedded, mutually disjoint, oriented $2$--spheres
	$$S_1,\ldots,S_n\subset X,$$
	such that $[S_1],\ldots,[S_n]$ form a basis of $H_2(X;\Integral)\cong \Integral^n$.
	For example, treating $X$ diffeomorphically as $(S^2\times D^2)^{\natural n}$,
	one may take the factor $2$--spheres
	in the boundary-connect summands $S^2\times D^2$.
	The normal vector bundle of any smoothly embedded $2$--sphere in $X$ is trivializable,
	since the intersection quadratic form vanishes on $H_2(X;\Integral)$.
	For each pair $(S_i,S_j)$ with $i<j$, 
	we obtain some region 
	$$\mathcal{N}_{ij}\subset X$$
	diffeomorphic to $(S^2\times D^2)^{\natural 2}$, by connecting some
	mutually disjoint tubular neighborhoods of $S_i$ and $S_j$ in $X$
	with some smoothly embedded $4$--dimensional $1$--handle
	in their complement closure.
	Fix some model barbell twist 
	$$\varphi\in\mathrm{Mod}(\mathcal{N},\partial\mathcal{N})$$
	on a model thickened barbell $\mathcal{N}$,
	with the model barbell spine $\Gamma=P_a\cup I\cup P_b$ (Definition \ref{BG_barbell_twist_def}).
	Note that there exists 
	some obvious orientation-preserving diffeomorphic involution of $\mathcal{N}_{ij}$
	which switches the oriented $2$--spheres $S_i$ and $S_j$ in $\mathcal{N}_{ij}$.
	Therefore, possibly after adjusting with such an involution,
	we can obtain an orientation-preserving diffeomorphism
	$\mathcal{N}\to\mathcal{N}_{ij}$,
	making sure that $[P_a]\wedge [P_b]$ maps to 
	the generator $[S_i]\wedge[S_j]$ of $\wedge^2H_2(\mathcal{N}_{ij};\Integral)\cong\Integral$.
	(We do not need to match up the model bells with the $2$--spheres, 
	although this can be done for sure.)
	Denote the implantation of $\varphi$ via $\mathcal{N}\to\mathcal{N}_{ij}\to X$
	as
	$$\hat{\varphi}_{ij}\in \mathrm{Mod}(X,\partial X).$$
	By the formula (\ref{Delta_BG_implanted_homological}) in Example \ref{BG_Delta_f_example}
	(see also Corollary \ref{implanted_homological}),
	we obtain
	$$\Delta\hat{\varphi}_{ij}=[S_i]\wedge[S_j]$$
	in $\wedge^2H_2(X;\Integral)$.
	Ranging over all $i,j\in\{1,\ldots,n\}$ with $i<j$,
	the alternating tensors	$[S_i]\wedge[S_j]$ form a free generating set for
	$\wedge^2H_2(X;\Integral)\cong\Integral^{n(n-1)/2}$.
	This proves $\mathrm{Im}(\Delta)\supset \wedge^2H_2(X;\Integral)$, and hence,
	$$\mathrm{Im}(\Delta)=\wedge^2H_2(X;\Integral),$$
	as desired.
		
	To prove $\mathrm{Ker}(\Delta)=\mathrm{Mod}(D^4,\partial D^4)$,
	we can pick some boundary $\pi_1$--nonbusting collection of 
	$n$ smoothly embedded, mutually disjoint $4$--dimensional $2$--handles
	$$(E_i\times D^2,\partial E_i\times D^2)\subset(X,\partial X),$$
	indexed by $i=1,\ldots,n$, such that
	$$V=X\setminus\left(\bigcup_{i=1}^n\left(\mathrm{int}(E_i)\times D^2\right)\right)$$
	is diffeomorphic to $D^4$ after corner smoothing along all $\partial E_i\times \partial D^2$
	(see Definition \ref{bnb_terms}).
	For example, treating $X$ diffeomorphically as $(S^2\times D^2)^{\natural n}$,
	one may take the tubular neighborhoods of the factor $2$--disks  
	in the boundary-connect summands $S^2\times D^2$.	
	It suffices to prove $\mathrm{Ker}(\Delta)\subset\mathrm{Mod}(D^4,\partial D^4)$.
	Suppose that $f\in\mathrm{Diff}^+(X,\partial X)$ satisfies $\Delta f=0$ in $\wedge^2H_2(X,\partial X)$.
	It follows that each $f(E_i)$ is homologous to $E_i$ in $H_2(X,E_i;\Integral)$,
	since the cellular $2$--cycle $f(E_i)-E_i$ is already null-homologous in $H_2(X;\Integral)$
	(see (\ref{Delta_f_def})).
	Since $X$ is simply connected, we infer that each $f(E_i)$ is homotopic to $E_i$
	in $X$ relative to $\partial E_i$, by the Hurewicz theorem.
	The inclusion of the union of all $E_i\times D^2$ 
	can be viewed as a boundary $\pi_1$--nonbusting smooth embedding of $n$ $4$--dimensional $2$--handles
	$$i_0\colon (\sqcup^n(D^2\times D^2),\sqcup^n(S^1\times D^2))\to (X,\partial X).$$
	Hence,
	$$i_1=f\circ i_0$$ 
	is another boundary $\pi_1$--nonbusting smooth embedding, 
	homotopic to $i_0$ relative to the attaching solid tori $\sqcup^n(S^1\times D^2)$.
	It follows that 
	$$i_0=h\circ i_1=h\circ f\circ i_0$$
	holds for some $h\in\mathrm{Diff}_0(X,\partial X)$,
	by our (ambient simply connected, boundary $\pi_1$--nonbusting) embedding uniqueness 
	for several $2$--handles (Corollary \ref{uniqueness_bnb_2hdls}).
	Set $g\in\mathrm{Diff}^+(X,\partial X)$ to be
	$$g=h\circ f.$$
	The above construction implies that $g$ preserves $V$ and
	fixes $X\setminus V$, so $g$ arises as the implantation of $g|_V$.
	Since $V$ is diffeomorphic to $D^4$, the boundary-fixing mapping class $[g]\in\mathrm{Mod}(X,\partial X)$
	lies in the factor subgroup $\mathrm{Mod}(D^4,\partial D^4)$.
	On the other hand, $h\in\mathrm{Diff}_0(X,\partial X)$
	represents the trivial boundary-fixing mapping class in $\mathrm{Mod}(X,\partial X)$,
	by definition.
	Therefore, we obtain
	$$[f]=[h^{-1}\circ g]=[h]^{-1}\cdot[g]=[g]$$
	in $\mathrm{Mod}(X,\partial X)$, implying that $[f]$ also lies in 
	the factor subgroup $\mathrm{Mod}(D^4,\partial D^4)$.
	This proves $\mathrm{Ker}(\Delta)\subset \mathrm{Mod}(D^4,\partial D^4)$, and hence,
	$$\mathrm{Ker}(\Delta)=\mathrm{Mod}(D^4,\partial D^4),$$
	as desired.
	
	Summarizing the above arguments,
	we have proved
	$\mathrm{Im}(\Delta)=\wedge^2H_2(X;\Integral)$ and $\mathrm{Ker}(\Delta)=\mathrm{Mod}(D^4,\partial D^4)$,
	and they imply that $\mathrm{Ker}(i_*)$ maps isomorphically onto $\wedge^2H_2(X;\Integral)$ under $\Delta$.
	Therefore, the assignment $[f]\mapsto (i_*[f],\Delta f)$ determines a group isomorphism
	$$\mathrm{Mod}(X,\partial X)\to \mathrm{Mod}(D^4,\partial D^4)\times \wedge^2H_2(X;\Integral).$$
	This proves the second statement in Lemma \ref{Mod_nS2D2_rel}.
	
	To prove the third statement in Lemma \ref{Mod_nS2D2_rel}, suppose that
	$$i_0,i_1\colon X\to D^4$$
	are orientation-preserving smooth embeddings,
	such that 
	the complementary regions
	$D^4\setminus i_0(\mathrm{int}(X))$ and $D^4\setminus i_1(\mathrm{int}(X))$
	are both diffeomorphic to $D^4\#(S^1\times D^3)^{\natural n}$.
		
	We claim that there exists some $h\in\mathrm{Diff}^+(D^4,\partial D^4)$,
	such that 
	$$h\circ i_1=i_0$$
	holds as smooth embeddings $X\to D^4$.
	In fact, we can identify $D^4$ as smoothly embedded in $S^4$
	(as a $4$--hemisphere). Then the regions $Z_0=S^4\setminus i_0(\mathrm{int}(X))$ 
	and $Z_1=S^4\setminus i_1(\mathrm{int}(X))$ are both diffeomorphic to $(S^1\times D^3)^{\natural n}$.
	Let $f\colon \partial Z_1\to \partial Z_0$ be the orientation-preserving diffeomorphism
	defined by the boundary restriction of 
	the orientation-preserving diffeomorphism $i_0\circ i_1^{-1}\colon i_1(X)\to i_0(X)$.
	Then, $f$ admits an extension as an orientation-preserving diffeomorphism,
	still denoted as
	$f\colon Z_1\to Z_0$ (Corollary \ref{extendability_1hdlbdy_corollary}).
	Moreover, we may assume the extended diffeomorphism
	to fix the smooth $4$--disk $S^4\setminus\mathrm{int}(D^4)$,
	possibly after adjusting the extension by 
	a boundary-fixing diffeotopy of $Z_0$.
	Set 
	$$h\colon D^4\to D^4$$
	to be $i_0\circ i_1^{-1}\colon i_1(X)\to i_0(X)$
	on $i_0(X)$, and to be
	$f|_{Z_1\cap D^4}\colon Z_1\cap D^4\to Z_0\cap D^4$
	on $Z_1\cap D^4$.
	Hence, $h$ is a boundary-fixing self-diffeomorphism of $D^4$,
	and satisfies $h\circ i_1=i_0$ on $X$,
	as claimed.
	
	Obtain some $h\in\mathrm{Diff}^+(D^4,\partial D^4)$ as above.
	Note that $\mathrm{Mod}(D^4,\partial D^4)$ is abelian,
	since any pair of boundary-fixing mapping classes of $D^4$ 
	admit representatives	with disjoint supports.
	For any $\phi\in\mathrm{Mod}(X,\partial X)$, we obtain
	$$(i_0)_*(\phi)=(h\circ i_1)_*(\phi)=[h]\cdot\left((i_1)_*(\phi)\right)\cdot[h]^{-1}=(i_1)_*(\phi),$$
	in $\mathrm{Mod}(D^4,\partial D^4)$.
	In other words, $i_0$ and $i_1$ induce
	identical implantation homomorphisms 
	$\mathrm{Mod}(X,\partial X)\to\mathrm{Mod}(D^4,\partial D^4)$.
	This proves the third statement of Lemma \ref{Mod_nS2D2_rel}.	 
\end{proof}

With the above preparation, we prove Theorem \ref{main_T} as follows.

Let $X$ be an oriented $4$--manifold diffeomorphic to $(S^2\times D^2)^{\natural n}$,
for some integer $n\geq0$.

To show the existence part in Theorem \ref{main_T},
we choose an orientation of $X$. 
Construct a group homomorphism 
$$
T\colon \wedge^2H_2(X;\Integral)\to\mathrm{Mod}(X,\partial X),
$$
as the composite homomorphism
\begin{equation}\label{T_construction}
\xymatrix{
\wedge^2H_2(X;\Integral) \ar[r] &
\mathrm{Mod}(D^4,\partial D^4)\times \wedge^2H_2(X;\Integral)
\ar[r] &
\mathrm{Mod}(X,\partial X),
}
\end{equation}
which is the factor inclusion followed by
the inverse of the isomorphism established in Lemma \ref{Mod_nS2D2_rel}.

It is easy to verify that that 
$T$ satisfies the three asserted properties in Theorem \ref{main_T}.
In fact, the homological formula follows directly from Corollary \ref{implanted_homological},
observing $\Delta T_\omega=\omega$ for all $\omega\in \wedge^2 H_2(X;\Integral)$.
The other two asserted properties in Theorem \ref{main_T},
involving $i\colon X\to D^4$ or $j\colon D^4\to X$,
are direct implications of Lemma \ref{Mod_nS2D2_rel} 
and the construction (\ref{T_construction}) of $T$,
if $i\colon X\to D^4$ is orientation-preserving with respect to
the chosen orientation of $X$.
The orientation-reversing case also follow easily
by composition with an orientation-reversing self-diffeomorphism of $D^4$,
which induces an automorphism of $\mathrm{Mod}(D^4,\partial D^4)$.
Therefore, the homomorphism $T$ as construction above
satisfy all the asserted properties in Theorem \ref{main_T}.

The uniqueness part in Theorem \ref{main_T} is implied immediately from the following simple observation.
For any $T\colon \wedge^2H_2(X;\Integral)\to \mathrm{Mod}(X,\partial X)$ satisfying 
all the declared properties in Theorem \ref{main_T},
we obtain a group homomorphism 
$$\mathrm{Mod}(D^4,\partial D^4)\times \wedge^2H_2(X;\Integral)$$
as defined by the assignment 
\begin{equation}\label{psi_omega}
(\psi,\omega)\mapsto T_\omega\cdot j_*(\psi),
\end{equation}
fixing some orientation of $X$ and
some orientation-preserving smooth embedding $j\colon D^4\to X$. 
Using the oriented doubling $Y=W_X=X\cup_{\partial X}(-X)$,
it is straightforward to check that 
the homomorphism (\ref{psi_omega}) is exactly inverse 
to the assignment $\phi\mapsto (i_*\phi,\Delta\phi)$ in Lemma \ref{Mod_nS2D2_rel}.
Therefore, $T$ has to take the form (\ref{T_construction}) as constructed above.

This completes the proof of Theorem \ref{main_T}. 

\section{Miscellaneous criteria for applications}\label{Sec-misc}
In this section, we mention two implications of Theorem \ref{main_T}.

The first implication (Proposition \ref{criterion_T})
provides a criterion for recognizing the boundary-fixing mapping class
$T_{\alpha\wedge\beta}$ in $\mathrm{Mod}(X,\partial X)$ as asserted in Theorem \ref{main_T},
for any $\alpha\wedge\beta$ in $\wedge^2 H_2(X;\Integral)$.

The second implication (Proposition \ref{transformation_T_general})
is a generalization of the transformation formula (Corollary \ref{main_corollary_transformation}).
In practice, it can be applied together with Corollary \ref{main_corollary_BG},
enabling us to generate the factor subgroup $\wedge^2 H_2(X;\Integral)$ 
with a collection of implanted barbell twists,
according to the natural isomorphism 
$\mathrm{Mod}(X,\partial X)\cong \mathrm{Mod}(D^4,\partial D^4)\times \wedge^2 H_2(X;\Integral)$
as in Corollary \ref{main_corollary_product}.

\begin{proposition}\label{criterion_T}
Let $X$ be an oriented $4$--manifold diffeomorphic to $(S^2\times D^2)^{\natural n}$,
for some integer $n\geq0$.
Let $T\colon \wedge^2 H_2(X;\Integral) \to \mathrm{Mod}(X,\partial X)$ 
be the homomorphism as declared in Theorem \ref{main_T}
Then, for any $\alpha,\beta\in H_2(X;\Integral)$,
$T_{\alpha\wedge\beta}\in \mathrm{Mod}(X,\partial X)$ is uniquely characterized by
the following properties.
\begin{itemize}
\item
The homological formula regarding $i_*(T_{\alpha\wedge\beta})$
as in Theorem \ref{main_T}
holds for the some oriented, connected, closed smooth $4$--manifold $Y$, 
and for some orientation-preserving smooth embedding $i\colon X\to Y$,
such that the induced linear homomorphism $i_*\colon H_2(X;\Integral)\to H_2(Y;\Integral)$
is injective.
\item
The triviality in $\mathrm{Mod}(D^4,\partial D^4)$ regarding $i_*(T_{\alpha\wedge\beta})$
holds for some smooth embedding $i\colon X\to D^4$,
such that $D^4\setminus i(\mathrm{int}(X))$ is diffeomorphic to $D^4\#(D^3\times S^1)^{\natural n}$.
\end{itemize}
\end{proposition}

\begin{proof}
Suppose that $\phi\in\mathrm{Mod}(X,\partial X)$ satisfies 
the asserted properties in Proposition \ref{criterion_T} regarding $\alpha\wedge\beta$.
In view of the natural isomorphism (Corollary \ref{main_corollary_product})
$$\mathrm{Mod}(X,\partial X)\cong\mathrm{Mod}(D^4,\partial D^4)\times \wedge^2 H_2(X;\Integral),$$
the triviality property ensures that 
$\phi$ lies in the factor subgroup $\wedge^2 H_2(X;\Integral)$.
The homological formula property
further ensures that $\phi$ 
agrees with $T_{\alpha\wedge\beta}$ as in Theorem \ref{main_T}.
Indeed, there are enough homology classes $\eta\in H_2(Y;\Integral)$
for distinguishing $\phi$ in $\mathrm{Mod}(X,\partial X)$ 
from any element in the factor subgroup $\wedge^2 H_2(X;\Integral)$ 
other than $\alpha\wedge\beta$,
because of the Poincar\'e duality in $Y$, and 
because of the injectivity assumption on $i_*\colon H_2(X;\Integral)\to H_2(Y;\Integral)$.
This identifies $\phi$ uniquely with $T_{\alpha\wedge \beta}$ in $\mathrm{Mod}(X,\partial X)$,
as asserted. 
\end{proof}

\begin{proposition}\label{transformation_T_general}
Let $X$ be an oriented $4$--manifold diffeomorphic to $(S^2\times D^2)^{\natural n}$,
for some integer $n\geq0$.
If $X'\subset\mathrm{int}(X)$ is a region diffeomorphic to $(S^2\times D^2)^{\natural m}$,
such that $X\setminus\mathrm{int}(X')$ 
is diffeomorphic to a boundary-connect sum 
of $(S^2\times S^1)^{\# m}\times D^1$ with $(S^2\times D^2)^{\natural (n-m)}$,
then, the following diagram of abelian group homomorphisms commute.
$$\xymatrix{
\wedge^2 H_2(X';\Integral) \ar[r]^-{T'} \ar[d]_{i'_*} & \mathrm{Mod}(X',\partial X') \ar[d]^{i'_*} \\
\wedge^2 H_2(X;\Integral) \ar[r]^-{T} & \mathrm{Mod}(X,\partial X).
}$$
Here, 
the horizontal arrows refer to the group homomorphisms as declared in Theorem \ref{main_T},
with respect to $X$ or $X'$;
the vertical arrows are induced by the inclusion $i'\colon X'\to X$.
\end{proposition}

\begin{proof}
View $X$ as smoothly embedded $\mathrm{int}(D^4)$,
such that $D^4\setminus\mathrm{int}(X)$ is diffeomorphic to $(D^3\times S^1)^{\natural n}$.
It follows that $D^4\setminus\mathrm{int}(X')$ is diffeomorphic to
$(D^3\times S^1)^{\natural m}$.

For any $\alpha',\beta'\in H_2(X';\Integral)$,
the implantation image of $T'_{\alpha'\wedge\beta'}$ in $\mathrm{Mod}(D^4,\partial D^4)$
with respect to the above inclusion is trivial, by Theorem \ref{main_T}.
It follows $i'_*(T'_{\alpha'\wedge\beta'})$ in $\mathrm{Mod}(X,\partial X)$
satisfies the triviality property in Proposition \ref{criterion_T}.
For any smooth embedding $X\to Y$,
the implantation image of $T'_{\alpha'\wedge\beta'}$ in $\mathrm{Mod}(Y,\partial Y)$
satisfies the homological formula as in Theorem \ref{main_T}.
Taking $Y$ to be the oriented doubling $W_X=X\cup_{\partial X}(-X)$,
it follows that $i'_*(T'_{\alpha'\wedge\beta'})$
satisfies the homological formula property in Proposition \ref{criterion_T},
regarding $i'_*(\alpha')\wedge i'_*(\beta')$ in $\wedge^2 H_2(X;\Integral)$.
By Proposition \ref{criterion_T}, 
we infer the following equality in $\mathrm{Mod}(X,\partial X)$.
$$i'_*\left(T'_{\alpha'\wedge\beta'}\right)=T_{i'_*\left(\alpha'\wedge\beta'\right)}.$$

As $\alpha',\beta'\in H_2(X;\Integral)$ are arbitrary,
the asserted commutative diagram holds.
\end{proof}

\section{Classification of barbell spines}\label{Sec-classification_spine}

In this section, we characterize and classify all the barbell spines 
in an orientable smooth $4$--manifold diffeomorphic to $(S^2\times D^2)^{\natural 2}$
(Corollaries \ref{classification_spine_GL} and \ref{characterization_spine}).
The classification essentially reduces to the classification
of $2$--sphere cut systems in $(S^2\times S^1)^{\# 2}$ 
(Theorem \ref{classification_spine} and Remark \ref{classification_spine_remark}),
which implies Corollaries \ref{classification_spine_GL} and \ref{characterization_spine}.

\begin{theorem}\label{classification_spine}
	Let $\mathcal{N}$ be a smooth $4$--manifold diffeomorphic to $(S^2\times D^2)^{\natural 2}$.
	
	Then, for every barbell spine $\Gamma=P_a\cup I\cup P_b$ of $\mathcal{N}$,
	the embedded disjoint union of $2$--spheres
	$P_a\sqcup P_b$ can be smoothly isotoped into $\partial\mathcal{N}$.
	Moreover, this relationship determines a well-defined, bijective correspondence
	betweeen 
	the barbell spines in $\mathcal{N}$, up to boundary-fixing diffeotopy of $\mathcal{N}$,
	and 
	the ordered pairs of 
	smoothly embedded, mutually disjoint, unionwise nonseparating, oriented $2$--spheres in $\partial \mathcal{N}$,
	up to diffeotopy of $\partial\mathcal{N}$.
\end{theorem}

\begin{remark}\label{classification_spine_remark}
	For any smooth $3$--manifold $M$ diffeomorphic to $(S^2\times S^1)^{\# n}$,
	any $n$--tuple of 
	smoothly embedded, mutually disjoint, unionwise nonseparating,
	transversely oriented $2$--spheres 
	$\mathcal{S}$ in $M$ (sometimes called a \emph{cut system})
	determines a graph-of-spaces decomposition of $M$, and hence
	a homotopically unique, continuous map $M\to \vee^n S^1$
	of $M$ onto the dual graph $\vee^n S^1$,
	which is the (standard) bouquet of $n$ labeled, oriented circles.
	The smooth isotopy classes of $\mathcal{S}$ in $M$ are classified exactly this way,
	and by the homotopy classes of continuous maps $M\to \vee^n S^1$,
	as a standard fact in $3$--manifold topology.
	The classifying set $[M,\vee^n S^1]$ naturally forms a homogeneous set modeled on
	the group 
	$$[\vee^n S^1,\vee^n S^1]\cong \mathrm{Out}\left(\pi_1(\vee^n S^1)\right)\cong\mathrm{Out}(\mathbb{F}_n),$$
	where $\mathbb{F}_n$ denotes the (standard) free group of rank $n$.
	The induced linear automorphic action of $\mathrm{Out}(\mathbb{F}_n)$ on the abelianization of $\mathbb{F}_n$
	determines a canonical surjective group homomorphism
	$\mathrm{Out}\left(\mathbb{F}_n\right)\to \mathrm{GL}(n,\Integral)$.
	The kernel of this homomorphism is analogous to 
	the Torelli subgroup in the mapping class group of an orientable closed surface.
	The kernel is nontrivial (and very complicated) for $n\geq3$, 
	but it is trivial for $n=0,1,2$.
	In particular, the following canonical group isomorphism is well-known 
	(see \cite[Chapter I, Proposition 4.5]{Lyndon--Schupp_book}).
	$$\mathrm{Out}(\mathbb{F}_2)\cong\mathrm{GL}(2,\Integral).$$	
	For further developments about the mapping class group of $(S^2\times S^1)^{\# n}$,
	see Brendle--Broaddus--Putman \cite{BBP_mcg}.
\end{remark}

\begin{corollary}[classifying barbell spines by homological bases]\label{classification_spine_GL}
	Let $\mathcal{N}$ be a $4$--manifold diffeomorphic to $(S^2\times D^2)^{\natural 2}$.
	Then,
	the boundary-fixing diffeotopy classes of barbell spines in $\mathcal{N}$ 
	correspond bijectively to the bases 
	of free abelian group $H_2(\mathcal{N};\Integral)\cong\Integral^2$, 
	such that $\Gamma=P_a\cup I\cup P_b$ corresponds to $[P_a],[P_b]$.
	In particular, they naturally form a homogeneous set modeled the group
	$\mathrm{GL}(H_2(\mathcal{N};\Integral))$, 
	which is isomorphic to $\mathrm{GL}(2,\Integral)$.
\end{corollary}

\begin{proof}
	This follows immediately from Theorem \ref{classification_spine} 
	and the natural isomorphism $H_2(\partial\mathcal{N};\Integral)\cong H_2(\mathcal{N};\Integral)$
	induced by the inclusion $\partial\mathcal{N}\to \mathcal{N}$,
	together with well-known facts as mentioned in Remark \ref{classification_spine_remark}.
\end{proof}

\begin{corollary}[characterizing a barbell spine]\label{characterization_spine}
	Let $\mathcal{N}$ be a smooth $4$--manifold diffeomorphic to $(S^2\times D^2)^{\natural 2}$.
	Then, a legally embedded barbell $\Gamma=P_a\cup I\cup P_b$ 
	forms a barbell spine of $\mathcal{N}$,
	if and only if
	$\Gamma$ sits in some parallel copy of $\partial \mathcal{N}$, up to boundary-fixing diffeotopy,
	and meanwhile, $[P_a],[P_b]$ form a basis of $H_2(\mathcal{N};\Integral)$.
\end{corollary}

\begin{proof}
	The ``only if'' direction is obvious from Theorem \ref{classification_spine},
	so it remains to prove the ``if'' direction. To this end, suppose that 
	$$\Gamma=P_a\cup I\cup P_b$$
	is a legally embedded barbell in $\mathcal{N}$,
	sitting in some parallel copy $M\subset\mathrm{int}(\mathcal{N})$
	of $\partial\mathcal{N}$, such that $[P_a],[P_b]$ form a basis of $H_2(\mathcal{N};\Integral)\cong\Integral^2$.
	
	We observe that the diffeotopy class of $\Gamma$ in $M$ 
	is uniquely determined by the basis $[P_a],[P_b]$ 
	together with the sides from which $I$
	approaches $P_a$ and $P_b$ in $M$ (with $2\times2=4$ possible types in total). 
	In fact, 
	the $2$--spheres $P_a$ and $P_b$ are necessarily unionwise nonseparating in $M$,
	and the basis $[P_a],[P_b]$ of $H_2(M;\Integral)\cong H_2(\mathcal{N};\Integral)$
	readily determines the smooth isotopy class of $P_a\sqcup P_b$ in $M$ (see Remark \ref{classification_spine_GL}).
	Moreover, any arc $I'\subset M$ coincident to $I$ near $\partial I'=\partial I$ 
	and intersecting $P_a\sqcup P_b$ only in $\partial I'$ is diffeotopic to $M$ fixing $P_a\sqcup P_b$.
	This is evident by considering $\mathrm{int}(I')$ and $\mathrm{int}(I)$ in the four-punctured $3$--sphere 
	$M\setminus(P_a\sqcup P_b)$. So, the above observation follows.
	
	By Corollary \ref{classification_spine_GL},
	we can find some barbell spine of $\mathcal{N}$, denoted as
	$$\Gamma'=P'_a\cup I'\cup P'_b,$$
	such that the basis $[P'_a],[P'_b]$ agrees with $[P_a],[P_b]$.
	By definition, 
	there exists some neatly embedded $3$--disk $(B,\partial B)\subset (\mathcal{N},\partial\mathcal{N})$,
	intersecting $\Gamma'$ transversely and only at a unique point in $\mathrm{int}(I')$,
	and witnessing a boundary-connect sum decomposition of $\mathcal{N}$ into summands of the form 
	$P'_a\times D^2$ and $P'_b\times D^2$ (Definition \ref{barbell_spine_def}).
	We can further assume $B$ to intersect $M$ in a parallel copy of $\partial B$.
	
	If the side-approaching type of $I$ to $P_a$ and $P_b$ 
	agrees with that of $I'$ to $P'_a$ and $P'_b$ in $M$
	(fixing an auxiliary orientation of $M$ to speak of signed sides),
	we can infer from the above observation that $\Gamma$ and $\Gamma'$ agree up to diffeotopy of $M$, 
	and hence, agree up to boundary-fixing diffeotopy of $\mathcal{N}$,
	proving that $\Gamma$ is a barbell spine.
	Otherwise, there still exists some (obvious) arc $\tilde{I}'\subset M$ 
	connecting $P'_a$ and $P'_b$ in $M$ without intersection in $\mathrm{int}(I')$,
	such that $I'$ has the same side-approaching type to $P'_a$ and $P'_b$
	as that of $I$ to $P_a$ and $P_b$,
	and intersects $B$ transversely at a unique point.
	Therefore, the reconstructed barbell
	$$\tilde{\Gamma}'=P'_a\cup\tilde{I}'\cup P'_b$$
	still forms a barbell spine of $\mathcal{N}$,
	and $\Gamma$ agrees with $\tilde{\Gamma}'$ up to boundary-fixing diffeotopy of $\mathcal{N}$.
	So, again, $\Gamma$ is a barbell spine,	proving the ``if'' direction.	
\end{proof}

The rest of this section is devoted to the proof of Theorem \ref{classification_spine}.

\begin{lemma}\label{classification_spine_basic}
	Let $\mathcal{N}$ be a smooth $4$--manifold diffeomorphic to $(S^2\times D^2)^{\natural 2}$.
	\begin{enumerate}
	\item For any barbell spine $\Gamma=P_a\cup I\cup P_b$ of $\mathcal{N}$,
	the embedded disjoint union of $2$--spheres
	$P_a\sqcup P_b$ can be smoothly isotoped into $\partial\mathcal{N}$.
	\item For any barbell spines $\Gamma_0$ and $\Gamma_1$ of $\mathcal{N}$,
	there exists some orientation-preserving self-diffeomorphism of $\mathcal{N}$
	which transforms $\Gamma_0$ to $\Gamma_1$.
	\end{enumerate}
\end{lemma}

\begin{proof}
	Obvious from our definition of a barbell spine (Definition \ref{barbell_spine_def}).
\end{proof}

\begin{lemma}\label{classification_spine_injective}
	Let $\mathcal{N}$ be a smooth $4$--manifold diffeomorphic to $(S^2\times D^2)^{\natural 2}$.
	If $\Gamma=P_a\cup I\cup P_b$ and $\Gamma'=P'_a\cup I'\cup P'_b$ are both barbell spines of $\mathcal{N}$,
	and if $P_a\sqcup P_b$ is smoothly isotopic to $P'_a\sqcup P'_b$ in $\mathcal{N}$,
	preserving the labeling and the bell orientations, 
	then $\Gamma$ and $\Gamma'$ agree	up to boundary-fixing diffeotopy of $\mathcal{N}$.
\end{lemma}

\begin{proof}
	Possibly after some boundary-fixing diffeotopy of $\mathcal{N}$, 
	we may assume $\Gamma'=P_a\cup I'\cup P_b$,
	such that $I'$ coincides with the arc $I$ 
	on some neighborhood of $\partial I'=\partial I$.
	Pick some point $q_a\in I$ near $P_a$, such that 
	the subarc between $P_a\cap \partial I$ to $q_a$
	coincides with a subarc of $I'$, and similarly, some point $q_b$ near $P_b$.
	Denote the subarcs on $I$ and $I'$ between $q_a$ and $q_b$ as $J$ and $J'$, respectively.
	Fix an orientation of $\mathcal{N}$.
		
	Since $\Gamma$ forms a barbell spine of $\mathcal{N}$,
	$\mathcal{N}\setminus(P_a\sqcup P_b)$ is clearly homotopy equivalent to 
	$(P_a\times S^1)\vee(P_b\times S^1)$ (see Definition \ref{barbell_spine_def}).
	Therefore, we observe
	$$\pi_1\left(\mathcal{N}\setminus(P_a\sqcup P_b),q_a\right)=\langle \mu_a,\mu_b\rangle\cong\mathbb{F}_2,$$
	where $\mu_a$ and $\mu_b$ are free generators represented by small meridian loops
	about the bells $P_a$ and $P_b$, with whisker paths connecting to $q_a$ along the bar $I$,
	respectively. 
	Similarly, using the whisker paths along the bar $I'$, we see that
	the same group is also generated by $\mu_a$ and $\mu'_b$,
	where $\mu'_b$ is a conjugate to $\mu_b$. 
	More precisely, we obtain
	$$\mu'_b=\tau\mu_b\tau^{-1},$$
	where $\tau\in\langle\mu_a,\mu_b\rangle$, 
	represented	by the closed path going first from $q_a$ to $q_b$ along $J'$,
	and then back to $q_a$ along $J$. 
	Note that we have used the assumption that $\Gamma'$ is a barbell spine of $\mathcal{N}$
	for the above claim, for otherwise we would not be able to 
	conclude that $\mu_a$ and $\mu'_b$ generate $\langle \mu_a,\mu_b\rangle$
	(which is false, indeed).
	It is an easy exercise of word processing in free groups 
	to check  
	$$\tau=\mu_a^k\mu_b^l,$$
	for some $k,l\in\Integral$, 
	so as to make $\langle \mu_a,\mu_b\rangle=\langle \mu_a,\mu'_b\rangle$.
	
	In the special case $k=l=0$, 
	$J$ is homotopic to $J'$ relative to $\partial J=\partial J'=\{q_a,q_b\}$
	in $\mathcal{N}\setminus(P_a\sqcup P_b)$.
	In this case, we observe that 
	$I$ is boundary-fixing diffeotopic to $I'$ in $\mathcal{N}$,
	fixing some neighborhood of $P_a\sqcup P_b$ all the time,
	since they are homotopic in $\mathcal{N}$
	relative to a common neighborhood of the endpoints.
	So, $\Gamma$ is boundary-fixing diffeotopic to $\Gamma'$ in $\mathcal{N}$,
	as desired. 
	
	For the general case ($k,l\in\Integral$), we observe that 
	rotating a tubular neighborhood of $P_a$ about $P_a$ for any number of full rounds
	can be realized with a boundary-fixing diffeotopy of $\mathcal{N}$.
	Using a tubular neighborhood of $P_a$ containing the subarc 
	of $I'$ from the endpoint on $P_a$ to $q_a$, 
	the above rotation will have the effect of modifying the arc $I'$,
	such that $k$ changes by adding any integer.
	Therefore, up to boundary-fixing diffeotopy of $\mathcal{N}$,
	we can adjust $k$, and similarly $l$, to be $0$ 
	without moving $P_a\sqcup P_b$.
	This reduces the general case to the special case $k=l=0$,
	which we have addressed.	
\end{proof}

\begin{lemma}\label{classification_spine_surjective}
	Let $\mathcal{N}$ be a smooth $4$--manifold diffeomorphic to $(S^2\times D^2)^{\natural 2}$.
	For any the ordered pair $S_a,S_b$ of 
	smoothly embedded, mutually disjoint, unionwise nonseparating, oriented $2$--spheres in $\partial \mathcal{N}$,
	there exists some barbell spine $\Gamma=P_a\cup I\cup P_b$ of $\mathcal{N}$,
	such that $P_a\sqcup P_b$ is smoothly isotopic to $S_a\sqcup S_b$ in $\mathcal{N}$,
	preserving the labeling and the bell orientations.
\end{lemma}

\begin{proof}
	In view of Remark \ref{classification_spine_remark} and Lemma \ref{classification_spine_basic},
	it suffices to show that every basis of $H_2(\mathcal{N};\Integral)$ 
	occurs as $[P_a],[P_b]$ for some barbell spine $\Gamma=P_a\cup I\cup P_b$ of $\mathcal{N}$.
	To this end, we observe some basic transformations for barbell spines 
	under the action of $\mathrm{Diff}^+(\mathcal{N})$, as follows.
	Fix a reference barbell spine  of $\mathcal{N}$, denoted as.
	$$\hat{\Gamma}=\hat{P}_a\cup\hat{I}\cup \hat{P}_b.$$
	Denote the corresponding basis of $H_2(\mathcal{N};\Integral)\cong\Integral^2$ as 
	$$(\alpha,\beta)=([P_a],[P_b]).$$
	
	The smooth $4$--manifold $\mathcal{N}$ can be viewed as a boundary-connect sum
	$$\mathcal{N}=\left(\hat{P}_a\times D^2\right)\natural\left(\hat{P}_b\times D^2\right).$$
	Since each summand $\hat{P}_{a/b}\times D^2$ admits 
	an obvious orientation-preserving involution,
	as realized by the product of reflections on the factors, 
	we observe that the following basis transformations
	can be realized by orientation-preserving diffeomorphisms of $\mathcal{N}$.
	$$\sigma_1(\alpha,\beta)=(\alpha,\beta)\cdot\left[\begin{array}{cc}-1&0\\0&1\end{array}\right],
	\mbox{ and }\sigma_2(\alpha,\beta)=(\alpha,\beta)\cdot\left[\begin{array}{cc}1&0\\0&-1\end{array}\right].$$
	We can also switch the summands in $(\hat{P}_a\times D^2)\natural (\hat{P}_b\times D^2)$
	by an orientation-preserving involution, realizing the following basis transformation.
	$$\sigma_3(\alpha,\beta)=(\alpha,\beta)\cdot\left[\begin{array}{cc}0&1\\1&0\end{array}\right].$$
	
	Viewing in another way, we may identify $\mathcal{N}$ alternatively as 
	$$\mathcal{N}=\left(D^3\setminus\mathrm{int}(Q_a\sqcup Q_b)\right)\times[-1,1],$$
	with corner smoothing,
	where $Q_a,Q_b$ are smoothly embedded, mutually disjoint $3$--disks in $D^3$.
	Because of Lemma \ref{classification_spine_basic}, 
	we can identify the bells of $\hat{\Gamma}$ as
	oriented, parallel nearby copies $\hat{P}_a$ and $\hat{P}_b$ 
	of $-\partial Q_a$ and $-\partial Q_b$, respectively, 
	in the $3$--dimensional slice $(\mathrm{int}(D^3)\setminus(Q_a\sqcup Q_b))\times 0$,
	and identify the bar $\hat{I}$ as sitting in the same slice.
	Since $D^3\setminus\mathrm{int}(Q_a\sqcup Q_b)$ is diffeomorphic
	to a $3$--sphere with three holes,
	we may replace one of the bells 
	$\hat{P}_b$ or $\hat{P}_a$ with a parallel copy of $\partial D^3$,
	and replace the bar $\hat{I}$ with some other bar connecting to that copy.
	By symmetry, this will give rise to other obvious barbell spines, which can be realized 
	as transformations of $\hat{\Gamma}$ under the action of $\mathrm{Diff}^+(\mathcal{N})$
	(Lemma \ref{classification_spine_basic}).
	This allows us to realize the following basis transformations. 
	$$\sigma_4(\alpha,\beta)=(\alpha,\beta)\cdot\left[\begin{array}{cc}1&-1\\0&-1\end{array}\right],
	\mbox{ and }
	\sigma_5(\alpha,\beta)=(\alpha,\beta)\cdot\left[\begin{array}{cc}-1&0\\-1&1\end{array}\right].$$
	
	Identifiying $\mathrm{GL}(H_2(\mathcal{N};\Integral))$ as $\mathrm{GL}(2,\Integral)$
	with respect to the reference basis $(\alpha,\beta)$,
	it is well-known that the composite transformations $\sigma_4\sigma_2$ and $\sigma_5\sigma_1$
	generate the special linear subgroup $\mathrm{SL}(2,\Integral)$,
	and together with $\sigma_3$, the transformations generate $\mathrm{GL}(2,\Integral)$.
	From these constructions, we see that every basis of $H_2(\mathcal{N};\Integral)$
	can be realized by some barbell spine of $\mathcal{N}$, as desired.
\end{proof}

With the above preparations, we finish the proof of Theorem \ref{classification_spine} as follows.

Let $\mathcal{N}$ be a smooth $4$--manifold diffeomorphic to $(S^2\times D^2)^{\natural 2}$.
For any barbell spine $\Gamma=P_a\cup I\cup P_b$ of $\mathcal{N}$,
there exists some ordered pair $S_a\sqcup S_b$ of 
smoothly embedded, mutually disjoint, unionwise nonseparating, oriented $2$--spheres in $\partial \mathcal{N}$,
such that $P_a\sqcup P_b$ is smoothly isotopic to $S_a\sqcup S_b$ in $\mathcal{N}$,
preserving labling and the bell orientations, by Lemma \ref{classification_spine_basic}.
The assignment $\Gamma\mapsto S_a\sqcup S_b$ is well-defined up to the asserted equivalence relations,
by Lemma \ref{classification_spine_injective}.
The injectivity of this assignment also follows from Lemma \ref{classification_spine_injective} (and Remark \ref{classification_spine_remark}).
The surjectivity of this assignment follows from Lemma \ref{classification_spine_surjective}.
Therefore, the assignment determines a bijective correspondence 
between the asserted objects up to the asserted equivalence relations.

This completes the proof of Theorem \ref{classification_spine}.

\bibliographystyle{amsalpha}

\end{document}